\documentclass[a4paper,10pt]{elsarticle}
\usepackage{amsmath}
\usepackage{amssymb}
\usepackage{amsfonts}
\usepackage{amsthm}
\usepackage{amsxtra}
\usepackage{rotating}
\usepackage{float}
\usepackage{epsfig}
\usepackage[dvips]{color}
\usepackage{array}
\usepackage{pifont}
\usepackage{multirow}
\usepackage{graphics}
\usepackage{fullpage}
\usepackage[mathscr]{euscript}
\usepackage[plainpages=false,pdfpagelabels]{hyperref}
\usepackage{natbib}

\theoremstyle{plain}
\newtheorem{theorem}{Theorem}[section]
\newtheorem{corollary}{Corollary}[section]
\newtheorem{lemma}{Lemma}[section]
\newtheorem{proposition}{Proposition}[section]
\theoremstyle{definition}
\newtheorem{definition}{Definition}[section]

\theoremstyle{remark}
\newtheorem{remark}{Remark}[section]

\def\8{\infty}

\def\Res{\mathop{\hbox{\rm Res}}}

\newcommand{\C}{\mathbb C}

\newcommand{\Z}{\mathbb Z}

\newcommand{\half}{
        {\lower0.00ex\hbox{\raise.6ex\hbox{\the\scriptfont0 1}
                           \kern-.5em\slash\kern-.1em\lower.45ex
                                     \hbox{\the\scriptfont0 2}}}}
\newcommand{\quarter}{
        {\lower0.00ex\hbox{\raise.6ex\hbox{\the\scriptfont0 1}
                           \kern-.5em\slash\kern-.1em\lower.45ex
                                     \hbox{\the\scriptfont0 4}}}}
\newcommand{\tquarter}{
        {\lower0.00ex\hbox{\raise.6ex\hbox{\the\scriptfont0 3}
                           \kern-.5em\slash\kern-.1em\lower.45ex
                                     \hbox{\the\scriptfont0 4}}}}
\newcommand{\eighth}{
        {\lower0.00ex\hbox{\raise.6ex\hbox{\the\scriptfont0 1}
                           \kern-.5em\slash\kern-.1em\lower.45ex
                                     \hbox{\the\scriptfont0 8}}}}
\newcommand{\othird}{
        {\lower0.00ex\hbox{\raise.6ex\hbox{\the\scriptfont0 1}
                           \kern-.5em\slash\kern-.1em\lower.45ex
                                     \hbox{\the\scriptfont0 3}}}}

\newcommand{\ddoAW}{\mathbb{D}}
\newcommand{\moAW}{\mathbb{M}}

\begin{document}

\begin{frontmatter}
\title{On a Family of Integrals that extend the Askey-Wilson Integral}

\author{M. Ito}
\address{School of Science and Technology for Future Life, Tokyo Denki University, Tokyo 101-8457, Japan}
\ead{mito@cck.dendai.ac.jp}

\author{N.S.~Witte}
\address{Department of Mathematics and Statistics,
University of Melbourne,Victoria 3010, Australia}
\ead{nsw@ms.unimelb.edu.au}

\begin{abstract}
We study a family of integrals parameterised by $ N = 2,3,\dots $ generalising the Askey-Wilson integral $ N=2 $ 
which has arisen in the theory of $q$-analogs of monodromy preserving deformations of linear 
differential systems and in theory of the Baxter $Q$ operator for the $ XXZ $ open quantum spin chain. 
These integrals are particular examples of moments defined by weights generalising the
Askey-Wilson weight and we show the integrals are characterised by various $ (N-1) $-th order linear $q$-difference 
equations which we construct. In addition we demonstrate that these integrals can be evaluated as a finite 
sum of $ (N-1) $ $ BC_{1} $-type Jackson integrals or $ {}_{2N+2}\varphi_{2N+1} $ basic hypergeometric functions. 
\end{abstract}


\end{frontmatter}

\section{Introduction}\label{Intro}
\setcounter{equation}{0}

The Askey-Wilson integral occupies an important position in the theory of $q$-integrals and
$q$-Selberg integrals because the underlying weight stands at the head of the Askey Table
of basic and ordinary hypergeometric function orthogonal polynomial systems. 
We recall the Askey-Wilson weight \cite{AW_1985} itself has four parameters 
$ \{a_1,\ldots,a_4\} $ with base $ q $
\begin{equation} 
  w(x) = w(x;\{a_1,\ldots,a_4\}) 
  = \frac{(z^2,z^{-2};q)_{\infty}}{\sin(\theta)\prod^{4}_{j=1}(a_jz,a_jz^{-1};q)_{\infty}} ,
    \quad z=e^{i\theta}, x=\tfrac{1}{2}(z+z^{-1})=\cos\theta \in \mathfrak{G}=(-1,1) .
\label{AWwgt}
\end{equation}
We also recall that the Askey-Wilson integral is defined by
\begin{equation*}
  I_{2}(a_1,a_2,a_3,a_4) 
  := \int_{\mathbb{T}} \frac{dz}{2\pi \sqrt{-1} z} \frac{(z^{\pm 2};q)_{\infty}}
                                   {\prod^{4}_{j=1}(a_{j}z^{\pm 1};q)_{\infty}} ,                                   
\end{equation*}
with $ |a_j|<1 $ for $ j=1,\ldots,4 $ and has the evaluation
\begin{equation*}
  I_{2}(a_1,a_2,a_3,a_4) 
   = 2\frac{(a_1a_2a_3a_4;q)_{\infty}}{(q;q)_{\infty}\prod_{1\leq j<k\leq 4}(a_{j}a_{k};q)_{\infty}} .
\end{equation*}
We employ the standard notations for the $q$ Pochhammer symbol, their products, the $q$-binomial coefficient
and the elliptic theta function
\begin{gather*}
  (a;q)_n = \prod^{n-1}_{k=0}(1-aq^{k}) ,\\ 
  (a;q)_{\infty} = \prod^{\infty}_{k=0}(1-aq^{k}) ,\\
  (a_1,a_2,a_3,\ldots ;q)_n = (a_1;q)_n(a_2;q)_n(a_3;q)_n\cdots ,\\
  (az^{\pm 1};q)_{\infty}=(az,az^{-1};q)_{\infty} ,\\
  \begin{bmatrix} n \\ m \end{bmatrix}_{q} = \frac{(q;q)_{n}}{(q;q)_{m}(q;q)_{n-m}} ,\quad m=0,1,\ldots, n \in \mathbb{N} ,\\
  \theta(z;q):=(z;q)_\8(qz^{-1};q)_\8 .
\end{gather*}
Also $ \mathbb{T}=\{z\in \mathbb{C}:|z|=1\} $ and we will invariably assume that
$ |q|<1 $, and other conditions on $ a_1, a_2, \ldots $ to ensure convergence of the
infinite products.

Our study treats a family of integrals that extend the Askey-Wilson integral in a direction not considered before.
The next member of such a family is (labelled as $ N=3 $ whereas the Askey-Wilson integral is labelled by $ N=2 $)
\begin{equation}
   I_{3}(a_1,\ldots,a_6) :=
   \int_{\mathbb{T}\cup\mathbb{H}}\frac{dz}{2\pi \sqrt{-1} z} \frac{z-z^{-1}}{z^{3/2}-z^{-3/2}}
   \frac{(z^{\pm 3};q^{3/2})_{\infty}}{\prod^{6}_{j=1}(a_j z^{\pm 1};q)_{\infty}} .
\label{N=3wgt}
\end{equation}
where the contour integral is the sum of a circular integral and a "tail" integral. This integrand possesses 
sequences of poles which accumulate at either $ z=\infty $ and $ z=0 $, so the contour $ \mathbb{T} $ can separate
these two sets if $ |a_j|<1 $ and circle the origin in an anti-clockwise direction, as occurs in the Askey-Wilson
case. Let us define $ a_j = q^{s_j} $, so that in the usual case $ |q|<1 $ we require also that $ {\rm Re}(s_j)>0 $
for $ j=1,\ldots, 6 $. However this integrand also has a branch cut, conventionally taken to be $ (-\infty,0) $
and so we take the contour path to be of a Hankel type with a tail $ \mathbb{H} $ starting at $ \infty e^{-\pi i} $
and finishing at $ \infty e^{+\pi i} $. This contribution involves a integral on $ [1,\infty) $ with an integrand
whose leading term as $ x \to \infty $ is
\begin{equation*}
   -\frac{1}{\pi}\exp\left( \frac{8\pi^2}{9\log q} \right)q^{\frac{3}{8}+\frac{1}{2}\sum^{6}_{j=1}s_j(s_j-1)}
    (1-x^{-2})x^{-3+\sum^{6}_{j=1}s_j}\prod^{6}_{j=1}\frac{(-q^{1-s_j}x^{-1};q)_{\infty}}{(-q^{s_j}x^{-1};q)_{\infty}} ,
\end{equation*}
where we have employed the asymptotic expansion of \cite{OD_1994}, which is valid for 
$ |{\rm arg}(a_j)| \leq \pi $ . This integral clearly converges under the additional condition that 
$ \sum^{6}_{j=1} {\rm Re}(s_j) < 2 $.
Clearly from the prefactors of the above expression the contribution from the tail
is exponentially small as $ q \to 1^{-} $.

Our aim, in fact, is to study the generalisations for $ N\geq 3 $, namely 
\begin{equation}
   I_{N} = I_{N}(a_1,\ldots,a_{2N}) :=
   \int_{\mathbb{T}\cup\mathbb{H}}\frac{dz}{2\pi \sqrt{-1} z} \frac{z-z^{-1}}{z^{N/2}-z^{-N/2}}
   \frac{(z^{\pm N};q^{N/2})_{\infty}}{\prod^{2N}_{j=1}(a_j z^{\pm 1};q)_{\infty}} .
\label{Witte_integral}
\end{equation}
Let us denote the integrand as $\Phi(z)$, which is the function defined as 
\begin{equation}
  \Phi(z) := \frac{z-z^{-1}}{z^{N/2}-z^{-N/2}}\frac{(z^{\pm N};q^{N/2})_\8}{\prod_{j=1}^{2N}(a_jz^{\pm 1};q)_\8} ,\quad
  I_N := \int_C \frac{dz}{2\pi \sqrt{-1}z}\Phi(z) , 
\label{Phidefn}
\end{equation}
where $C$ is as given earlier and depends on $N$ being odd or even. Clearly there is a distinction to be made
between the even and odd cases of $ N $ and the tail contribution is absent for $ N $ even. The tail
contribution involves a integral on $ [1,\infty) $ with an integrand whose asymptotic expansion as $ x \to \infty $
is now
\begin{align*}
&    \frac{1}{\pi}\sin \left(\tfrac{1}{2}\pi N \right) \exp\left( \frac{\pi^2(N^2-1)}{3N\log q} \right)q^{\frac{N}{8}+\frac{1}{2}\sum^{2N}_{j=1}s_j(s_j-1)}
    (1-x^{-2})x^{-N+\sum^{2N}_{j=1}s_j}\prod^{2N}_{j=1}\frac{(-q^{1-s_j}x^{-1})_{\infty}}{(-q^{s_j}x^{-1})_{\infty}} \\
& \quad    \times \exp \left\{ \vphantom{\frac{1}{k\sinh \left(\displaystyle \frac{2\pi^2k}{\displaystyle \log q} \right)}
              \left[ -\sum^{2N}_{j=1} \right]}
\sum^{\infty}_{k=1} \frac{1+(-1)^k}{k\sinh \left(\frac{\displaystyle 2\pi^2k}{\displaystyle N\log q} \right) }
	\cos \left(2\pi k\frac{\log x}{\log q} \right) \right.  \notag \\
& \qquad     \left.  
	+  \sum^{\infty}_{k=1} \frac{1}{k\sinh \left(\displaystyle \frac{2\pi^2k}{\displaystyle \log q} \right)}
              \left[ -\cos \left(2\pi k\frac{\log x}{\log q}\right)\sum^{2N}_{j=1}\cos 2\pi ks_j
                     +\sin \left(2\pi k\frac{\log x}{\log q}\right)\sum^{2N}_{j=1}\sin 2\pi ks_j \right] \right\} , \notag 
\end{align*}
where again we have employed the asymptotic expansion of \cite{OD_1994}, which is valid for 
$ |{\rm arg}(a_j)| \leq \pi $ . This integral clearly converges under the additional condition that 
$ \sum^{2N}_{j=1} {\rm Re}(s_j) < N-1 $.
There also exist a class of somewhat more general integrals, of which the above are particular cases,
which are the general solutions of a common system of linear divided-difference equations involving
the Askey-Wilson divided-difference operator. We shall argue that these integrals are a natural extension
of the Askey-Wilson system, on the basis of the theory underlying these integrals.
 
The weights \eqref{Witte_integral} were introduced in an insightful observation made by Ismail et al \cite{ILR_2004}
in their investigation of second order Sturm-Liouville equations in the Askey-Wilson divided-difference operators.
Their work demonstrated that a particular generalisation of the Askey-Wilson weight has profound applications
in the theory of quantum integrable models soluble using the Bethe-Ansatz, in this case the XXZ spin chain.
The one-dimensional $U_q(sl_2({\bf C}))$-invariant XXZ model of spin $\frac{1}{2}$ of a size $2N$ with the
open (Dirichlet) boundary conditions can be diagonalised using the {\it Quantum Inverse Scattering Method}
\cite{Sk_1988, KS_1991}, which requires the solutions $ \lambda_k $, $ k=1,\ldots, n $
of the following set of 
nonlinear equations, known as the {\it Bethe Ansatz Equations} (given here for the inhomogeneous chain)
\begin{equation}
   \prod_{l=1}^{2N}\frac{\sin(\lambda_k+s_l\eta)}{\sin(\lambda_k-s_l\eta) }
  = \prod_{j\neq k,j=1}^n \frac{\sin(\lambda_k+\lambda_j+\eta)\sin(\lambda_k-\lambda_j+\eta)}
                               {\sin(\lambda_k+\lambda_j-\eta)\sin(\lambda_k-\lambda_j-\eta)},
    \quad 1 \leq k \leq n, \quad n\in \mathbb{N} .
\label{BAeqn}
\end{equation}
Here the $s_l$'s are $ 2N $ complex numbers (the spins at site $l$) and $ q = e^{2{\rm i}\eta} $ is fixed by 
the XXZ anisotropy parameter $ \Delta = \frac{1}{2}(q+q^{-1}) $. However in \cite{ILR_2004} they construct the
Bethe Ansatz equations from entirely analytic arguments without reference to the quantum spin chain. Their key
result is the following statement.
\begin{theorem}[\cite{ILR_2004}]
Let $ a_l=q^{-s_l}=e^{-2{\rm i}\eta s_l}$ for $ 1 \leq l \leq 2N $ and $ \vec{a} = (a_1, \ldots, a_{2N}) $.
Consider the second order $q$-Sturm-Liouville equation 
\begin{equation}
  \Pi(x)\ddoAW^2_x p(x) + \Phi(x)(\moAW_x\ddoAW_x p)(x) = r(x)p(x) ,
\label{qS-Leqn}
\end{equation}
where $ \ddoAW_x, \moAW_x $ are the Askey-Wilson divided-difference operators defined subsequently by
\eqref{DDoperator:a}, \eqref{DDoperator:b} and the coefficients $ \Pi(x;\vec{a}), \Phi(x;\vec{a}) $ are
polynomials in $ x $ defined by the relations
\begin{gather*}
  \Pi(x;\vec{a}) = \frac{1}{w(x;\vec{a})}\moAW_x w(x;q^{1/2}\vec{a}) ,
\\
  \Phi(x;\vec{a}) =  \frac{1}{w(x;\vec{a})}\ddoAW_x w(x;q^{1/2}\vec{a}) ,
\end{gather*}
with the weight function $ w $ given by
\begin{equation} 
   w(x;\vec{a} ) := \frac{(e^{{\rm i}N\theta},e^{-{\rm i}N\theta};q^{\frac{N}{2}})_\infty}
                          {\sin \frac{N\theta}{2}\prod_{j=1}^{2N}(a_j e^{{\rm i}\theta},a_j e^{-{\rm i}\theta};q)_\infty} .
\label{qS-Lweight}
\end{equation}
Then polynomial solutions $ p(x) $ of \eqref{qS-Leqn}, with arbitrary $ {\rm deg}_x p=n\geq 0 $, possess
zeros at $ x_k=\cos 2\lambda_k $ where the phases $ \lambda_k $ satisfy the Bethe Ansatz equations \eqref{BAeqn}.
\end{theorem}

The weight \eqref{qS-Lweight} is precisely the object of our study. The subject matter that is underlying this
result is the Heine-Stieltjes electrostatic interpretation of polynomial solutions to second order Sturm-Liouville
equations generalised from differential operators to divided-difference operators on quadratic lattices. In the
original setting the Sturm-Liouville equation, see for example Szeg\H{o}'s book {\it Orthogonal Polynomials}
\cite{ops_Sz} (p. 151, Chapter VI, Section 6.8), is
\begin{equation}
  \Pi(x)\frac{d^2 p}{dx^2}+\Phi(x)\frac{dp}{dx}+r(x)p = 0, 
\label{H-S} 
\end{equation}
where $\Pi(x)$ and $\Phi(x)$ are given polynomials of degrees $N$ and $N-1$. Heine \cite{HeineKF_1961}
(volume 1, pp. 472-479) proved that given any nonnegative integer $n$, there exists at most $ \binom{N+n-2}{N-2} $
choices of the polynomial $r(x)$ such that \eqref{H-S} has a polynomial solution with degree $ n $. Stieltjes
\cite{St_1885}, \cite{StieltjesOC_1993} (volume 1, pp. 434-439) continued this line of inquiry and 
showed that if one assumes that $\Pi(x)$ and $\Phi(x)$ have only real and simple zeros and their zeros interlace, 
then there are precisely $ \binom{N+n-2}{N-2} $ polynomials $r(x)$ which will permit a polynomial solution of 
\eqref{H-S} with degree $ n \in \mathbb{N} $. Whether or not such polynomial sequences form an orthogonal system is 
not essential for our purposes.
In the Sturm-Liouville problem where the coefficients $\Pi(x), \Phi(x), r(x)$ of \eqref{H-S} are polynomials in 
$x$ with the degrees $ {\rm deg}\,\Pi = 1+{\rm deg}\, \Phi = 2+{\rm deg}\, r = N \geq 2 $
a natural question arises: does replacing $\Pi_2, \Phi_1$ and $r_0$ (i.e. the classical case) by $\Pi_{N} $, 
$ \Phi_{N-1} $, and $r_{N-2}$ in \eqref{H-S} lead to more general orthogonal polynomials? The answer was given 
in the negative with Bochner's Theorem \cite{Bo_1929}. Another natural question also arises: is there a $q$ 
analogue of Bochner's Theorem? This again was answered in the negative by Gr\"unbaum and Haine \cite{GH_1996} who
proved that the only orthogonal polynomial solutions to the $q$-Sturm-Liouville equation after the replacements 
$(\Pi_2, \Phi_1,r_0) \to (\Pi_{N}, \Phi_{N-1}, r_{N-2})$ are the Askey-Wilson polynomials or special and limiting 
cases of them.

Conforming with standard notation \cite{GR_2004} we define the basic hypergeometric function 
$ {}_{r+1}\varphi_{r} $ by the series 
\begin{equation}
   {}_{r+1}\varphi_{r}\left[
                   \begin{array}{cccc}
                     a_{1}, & a_{2}, & \ldots, & a_{r+1} \\
                     b_{1}, & b_{2}, & \ldots, & b_{r} 
                   \end{array} ; q,z
                   \right] 
  = \sum^{\infty}_{n=0} \frac{(a_{1},a_{2},\ldots,a_{r+1};q)_n}{(q,b_{1},b_{2},\ldots,b_{r};q)_n} z^n ,
\label{phiDefn}  
\end{equation}
which is convergent for $ |z| < 1 $. 
The very-well-poised basic hypergeometric function $ {}_{r+1}W_{r} $ is a specialisation 
of the above
\begin{equation}
   {}_{r+1}W_{r}(a_1;a_4,a_5,\ldots,a_{r+1};q,z)
  = {}_{r+1}\varphi_{r}\left[
                   \begin{array}{cccccc}
                     a_{1}, & q\sqrt{a_{1}}, & -q\sqrt{a_{1}}, & a_4, & \ldots, & a_{r+1} \\
                     \sqrt{a_{1}}, & -\sqrt{a_{1}}, & qa_{1}/a_4, & \ldots, & qa_{1}/a_{r+1} & \\
                   \end{array} ; q,z
                   \right]  
\end{equation}
so that $ qa_{1} = a_{2}b_{1} = a_{3}b_{2} = \cdots = a_{r+1}b_{r} $ and 
$ a_{2} = q\sqrt{a_{1}} $, $ a_{3} = -q\sqrt{a_{1}} $.
The $ {}_{r+1}\varphi_{r} $  or $ {}_{r+1}W_{r} $ functions may be balanced, whereby $ z=q $ and
$ \prod^{r}_{j=1}b_{j} = q\prod^{r+1}_{j=1}a_j $. The bilateral basic hypergeometric function $ {}_{r}\psi_{r} $ is defined by
\begin{equation}
{}_{r}\psi_{r}
\left[\!\!\begin{array}{c}
		a_{1},a_{2},\ldots, a_{r}\\
		b_{1},b_{2},\ldots, b_{r} 
	\end{array}\!\!
	;q,z\right]
  := \sum_{n =-\infty }^{\infty}
	\frac{(a_1,a_2,\ldots,a_r;q)_n}
	{(b_1,b_2,\ldots,b_r;q)_n}z^n . 
\label{psiDefn}
\end{equation}
This new family of integrals differs from other generalisations of the Askey-Wilson integral, such as the
Nassrallah-Rahman integral \cite{NR_1985}, the integral representations of the very-well poised $ {}_8W_{7} $
or of the combination of two balanced, very-well-poised $ {}_{10}W_{9} $ functions \cite{Ra_1986b}, or of
the very-well-poised bilateral $ {}_{8}\psi_{8} $ \cite{Ra_1996} and
the elliptic $ {}_{12}V_{11} $ functions \cite{vBRS_2007,vBR_2009}.

The plan of our study is as follows: In \S \ref{OPS_NULonQuadLattice} we give a brief account of some definitions
and results for the discrete calculus on the $q$-quadratic lattice - the master case of the quadratic lattices in
the classification of Magnus \cite{Ma_1988} - which will require the introduction of the Askey-Wilson operators.
We then formulate a moment problem on the $q$-quadratic lattice and its associated orthogonal polynomial system, and 
this yields certain integral representations for the moments which form the subject of our study. The notion of a 
semi-classical weight is central here and we give a general, linear recurrence relation for the moments \eqref{Mrecur}
of such a system, first derived in \cite{Wi_2010a}. In \S \ref{differenceIN} we derive a $ (N-1) $-th order linear
$q$-difference equation for $ I_{N} $ using two different methods (see \eqref{I_q-DifferenceEqn} and \eqref{eq:g02-1}
respectively) one of which follows from the theory of \S \ref{OPS_NULonQuadLattice}. In addition we derive a
$ (N-1) $-th order mixed, linear partial $q$-difference equation for $ I_{N} $ (see \eqref{eq:IN-c} and
\eqref{eq:coeff-c}) and show that a vector of $q$-shifted integrals $ I_{N} $ satisfies a first order, linear matrix 
$q$-difference system (see \eqref{eq:1st-diff}). Finally we give in \S \ref{BC1evalIN} an explicit evaluation of
$ I_{N} $ as a sum of $ (N-1) $ $ BC_{1}$-type Jackson integrals or $ {}_{2N+2}\varphi_{2N+1} $ basic hypergeometric
functions (see \eqref{eq:IN03}).

\section{The Moment Problem for Semi-classical Weights on the \texorpdfstring{$q$}{q}-Quadratic Lattice}\label{OPS_NULonQuadLattice}
\setcounter{equation}{0}

Our theory will treat the case of the $q$-quadratic lattice and the Askey-Wilson
divided-difference calculus, and in order to simplify the description and to conform to
convention we will employ the canonical, that is to say the centred and symmetrised forms
of the lattice and the divided-difference operators.
Let us define the base $ q=e^{2i\eta} $ although we will not restrict ourselves to 
$q$-domains such as $ 0< \Re(q) < 1 $ except to avoid special degenerate 
cases and to ensure convergence.
Consider the projection map from unit circle $ \mathbb{T} \to [-1,1] $
\begin{equation*}
  z=e^{i\theta}, \quad \theta \in [-\pi,\pi) ,\qquad x=\tfrac{1}{2}(z+z^{-1})=\cos\theta \in [-1,1] . 
\end{equation*}
We denote the unit circle by $ \mathbb{T} $ and the unit open disc by $ \mathbb{D} $.
The inverse of the projection map defines a two-sheeted Riemann surface, one of which
corresponds to the interior of the unit circle and the other to the exterior. Thus we take
$ x $-plane to be cut along $ [-1,1] $ and will usually give results for the second sheet
i.e. when $ |x|\to \infty $ as $ z\to \infty $. To any function of $x$, say $f(x)$, we denote
the corresponding function of $z$ by $\tilde{f}(z) = f(x)$.
Define the shift operators $ E^{\pm}_{x} $ by
\begin{equation}
  E^{\pm}_{x}f(x) = f(\tfrac{1}{2}[q^{\pm 1/2}z+q^{\mp 1/2}z^{-1}]) ,
\label{DD_xshift}
\end{equation} 
and set $ y_{\pm} = E^{\pm}_{x} x $. 
This implies that 
\begin{gather*}
   y_{+}+y_{-} = (q^{1/2}+q^{-1/2})x = 2\cos\eta\; x ,
\\
   \Delta y := y_{+}-y_{-} = \tfrac{1}{2}(q^{1/2}-q^{-1/2})(z-z^{-1}) = -2\sin\eta \sin\theta ,
\\
   \Delta y^2 = (q^{1/2}-q^{-1/2})^2(x^2-1) ,
\\
   y_{+}y_{-} = x^2+\tfrac{1}{4}(q^{1/2}-q^{-1/2})^2 = x^2-\sin^2\eta .
\end{gather*} 
The Askey-Wilson divided-difference operators are defined as
\begin{align}
  \ddoAW_{x}f(x) & = \frac{E^{+}_{x}f(x)-E^{-}_{x}f(x)}{E^{+}_{x}x-E^{-}_{x}x} ,
\label{DDoperator:a} \\
  \moAW_{x}f(x) & = \tfrac{1}{2}\left[ E^{+}_{x}f(x)+E^{-}_{x}f(x) \right] .
\label{DDoperator:b}
\end{align} 
One has a parameterisation of the non-uniform $q$-quadratic lattice ($ \theta=2s\eta $, $ s\in \mathbb{Z} $)
\begin{gather*}
  x(s) = x_s = \tfrac{1}{2}(q^{s}+q^{-s}) = \cos(2\eta s) ,
 \\
  y_{\pm}(s) = \tfrac{1}{2}(q^{s\pm 1/2}+q^{-s\mp 1/2}) = \cos(2\eta [s\pm \tfrac{1}{2}]) .
\end{gather*}
Here the direct lattice is 
$ \mathfrak{G}[x(a)=\tfrac{1}{2}(a+a^{-1})] = \{\tfrac{1}{2}(q^{r/2}a+q^{-r/2}a^{-1}): r\in 2\Z\} $.
For $ |q|=1 $ and $ \eta $ not a rational multiple of $ \pi $ then the lattice densely fills the interval 
$ [-1,1] $. However if $ \eta $ is a rational multiple of $ \pi $ then one has the root of unity case 
$ q^N=1 $ and a finite lattice. We will restrict ourselves to the class of functions $ f $ where the sum 
over the $q$-quadratic lattice is defined and is equal to the Riemann-Stieltjes integral 
\begin{equation*}
	\int_\mathfrak{G}\ddoAW x\,f(x) := \sum_{s\in \mathbb{Z}} \Delta y(x_s)f(x_s) = \frac{\sin\eta}{\eta}\int^{1}_{-1}dx\,f(x) .
\end{equation*}

Let $ \{\phi_n(x;a)\}^{\infty}_{n=0} $ be a polynomial basis of $ L^2(w(x)\ddoAW x,\mathfrak{G}) $, 
where the support is $ \mathfrak{G} = \{E^{+k}_x x: k \in 2\Z \} $
and $ a $ denotes the parameter characterising the lattice. The Heine function basis 
\begin{equation*}
   \phi_n(x;a) = (az^{\pm 1};q)_n ,\quad a\neq 0, \infty ,\: n \in \mathbb{Z}_{\geq 0} ,
\end{equation*}
has the following properties:
\begin{enumerate}
\item[(1)]
$ \phi_n $ is of exact degree $ n $ with
\begin{equation*}
   \phi_n(x;a) = (-2a)^nq^{\frac{1}{2}n(n-1)}x^n + {\rm O}(x^{n-1}) ,
\end{equation*}  
\item[(2)]
$ \ddoAW_x $ is an exact lowering operator in this basis 
\begin{equation*}
   \ddoAW_{x} \phi_n(x;a) = -2a\frac{q^n-1}{q-1} \phi_{n-1}(x;q^{1/2}a) ,
\end{equation*}
\item[(3)]
$ \phi_n $ satisfies the linearisation formula
\begin{equation*}
   x\phi_n(x;a) = -\frac{1}{2aq^n}\phi_{n+1}(x;a)+\tfrac{1}{2}(aq^n+a^{-1}q^{-n})\phi_{n}(x;a) .
\end{equation*}
\end{enumerate}
A general solution to the first two requirements above is the following product expression
\begin{equation*}
   \phi_n(x;a) = (-2a)^nq^{\frac{1}{2}n(n-1)}\prod^{n-1}_{k=0}[x-(E^{+}_{x})^{2k} x(a)] ,
\end{equation*}
where the base point $ x(a) $ is parameterised by $ a $. We note that the limit $ \lim_{n\to\infty}\phi_n(x;a) $
exists for all $ a,x \in \C $ and $ |q|<1 $, and we denote this by $ \phi_{\infty}(x;a) $.
This enables us to define an analytic continuation of $ \phi_n(x;a) $ beyond positive integer $ n $
\begin{equation}
  \phi_r(x;a) = \frac{(az^{\pm 1};q)_{\infty}}{(aq^rz^{\pm 1};q)_{\infty}} ,
\end{equation} 
for all $ r\in\C $ but subject to $ |q|<1 $. We record for subsequent use the action of the divided-difference
operators on the basis functions
\begin{align}
   \ddoAW_{x} \phi_r(x;a)
 & = 2a\frac{1-q^r}{q-1}\phi_{r-1}(x;q^{1/2}a) ,
\label{DDO_basis:a}
 \\
   \moAW_{x} \phi_r(x;a)
 & = \tfrac{1}{2}(1+q^{-r})\phi_r(x;q^{1/2}a)+\tfrac{1}{2}(1-q^{-r})(1-a^2q^{2r-1})\phi_{r-1}(x;q^{1/2}a) , 
\label{DDO_basis:b}
 \\
   \ddoAW_{x} \phi_{\infty}(x;a)
 & = \frac{2a}{q-1}\phi_{\infty}(x;q^{1/2}a) ,
\label{DDO_basis:c}
 \\
   \moAW_{x} \phi_{\infty}(x;a)
 & = \phi_{\infty}(x;q^{-1/2}a) .
\label{DDO_basis:d}
\end{align}

Our study requires a revision of a number of standard results in orthogonal polynomial theory \cite{ops_Sz}, 
\cite{ops_Fr}, \cite{Ismail_2005} so we briefly recount the formulation given in \cite{Wi_2010a}. The latter 
work treats the theory of $\delta$-analogs and $q$-analogs for monodromy preserving deformations of linear differential systems
which is not of direct concern here, however much of the foundational material is still relevant.
A number of these results have been found by earlier studies, most notably by the Soviet school of Nikiforov,
Suslov and Uvarov \cite{NSU_1984,NS_1985,NSU_1986,NS_1986a,NS_1986b}, who were primarily concerned with 
hypergeometric type orthogonal polynomial system on non-uniform lattices and in the 1988 work of Magnus \cite{Ma_1988}. 
Consider the general {\it orthogonal polynomial system} $ \{p_n(x)\}^{\infty}_{n=0} $ defined by the orthogonality 
relations
\begin{equation}
  \int_{\mathfrak{G}} \ddoAW x\; w(x)p_n(x)\,\phi_m(x;a) = \begin{cases} 0, \quad 0\leq m<n \\ h_n(a), \quad m=n \end{cases},
  \quad n \geq 0 ,\quad a \neq 0, \infty ,
\label{ops_orthog}
\end{equation}
with $ \mathfrak{G} $ denoting the support of the {\it weight} $ w(x) $. 
Our system of orthogonal polynomials have a distinguished singular point at $ x=\infty $ and possess
expansions about this point which can characterise solutions uniquely. This is related to the fact that 
orthogonal polynomials are the denominators of single point Pad\'e approximants and that point is conventionally 
set at $ x=\infty $. We give special notation for the coefficients of $ x^{n} $ and $ x^{n-1} $ in $ p_n(x) $,
\begin{equation*}
   p_n(x) = \gamma_n x^{n} + \gamma_{n,1}x^{n-1} + \ldots, \quad n \geq 0 ,\quad \gamma_{0,1}=0 .
\label{ops_poly}
\end{equation*}
A consequence of the orthogonality relation is the three term recurrence relation
\begin{equation*}
   a_{n+1}p_{n+1}(x) = (x-b_n)p_n(x) - a_np_{n-1}(x), \quad n \geq 0,
\label{ops_threeT}
\end{equation*}
and we consider the sequence of orthogonal polynomials generated from the initial values $ p_{-1} = 0 $ and 
$ p_0 = \gamma_0 $. The three term recurrence coefficients are related to the leading and sub-leading
polynomial coefficients by \cite{ops_Sz}, \cite{ops_Fr}
\begin{equation*}
   a_n = \frac{\gamma_{n-1}}{\gamma_n}, \quad
   b_n = \frac{\gamma_{n,1}}{\gamma_n}-\frac{\gamma_{n+1,1}}{\gamma_{n+1}}, \quad n \geq 1 ,\quad
   b_0 = -\frac{\gamma_{1,1}}{\gamma_{1}} .
\label{ops_coeffRn}
\end{equation*}

The orthogonality relation (\ref{ops_orthog}) is derived from the linear functional on the space of polynomials
\begin{equation*}
    {\mathcal L}: p\in \Pi \mapsto \C ,
\label{ops_linearF}
\end{equation*}
and we employ our basis polynomials as an expansion basis although not necessarily with the same parameter 
used in \eqref{ops_orthog}
\begin{equation*}
   p_n(x) = \sum^{n}_{k=0} c_{n,k}(b)\phi_k(x;b)
  \quad n \geq 0 .
\label{ops_expand}
\end{equation*}
Consequently we define the {\it moments} $ \{m_{j,k}\}_{j,k=0,1,\ldots,\infty} $ of the weight as the 
action of this functional on products of the basis polynomials, defined as
\begin{equation}
   m_{j,k} = m_{j,k}(a,b) := \int_{\mathfrak{G}} \ddoAW x\; w(x)\,\phi_{j}(x;a)\phi_{k}(x;b) ,
  \quad j,k \geq 0 .
\label{ops_moment}
\end{equation}
The relevance of the moment problem to the associated orthogonal polynomial system arises through the 
{\it moment determinants}
\begin{equation*}
   \Delta_n(a,b) := \det[ m_{j,k}(a,b) ]_{j,k=0,\ldots,n-1}, \quad n\geq 1,
   \quad \Delta_0 := 1 ,
\label{ops_Hdet}
\end{equation*}
and for $  n\geq 1, j=0,\ldots n-1 $
\begin{equation*}
   \Sigma_{n,j}(a,b) := \det\left(
               \begin{array}{ccccccc}
               m_{0,0}(a,b)   & \cdots & m_{0,j-1}(a,b)   & []     & m_{0,j+1}(a,b)   & \cdots & m_{0,n}(a,b)   \\
               \vdots    & \vdots & \vdots      & \vdots & \vdots      & \cdots & \vdots    \\
               m_{n-1,0}(a,b) & \cdots & m_{n-1,j-1}(a,b) & []     & m_{n-1,j+1}(a,b) & \cdots & m_{n-1,n}(a,b) \\
               \end{array} \right) ,
\label{ops_Sdet}
\end{equation*}
defined in terms of the moments above. Obviously $ \Delta_n = \Sigma_{n,n} $ and we set $ \Sigma_{0,0} := 0 $.
The expansion coefficients $ \{c_{n,j}\}^{n}_{j=0} $ are given in terms of these determinants by
\begin{equation*}
   c_{n,j}(b) = (-)^{n+j}h_n(a)\frac{\Sigma_{n,j}(a,b)}{\Delta_{n+1}(a,b)} ,\quad
   c_{n,n}(b) = h_n(a)\frac{\Delta_{n}(a,b)}{\Delta_{n+1}(a,b)} .
\label{ops_expCff}
\end{equation*} 
It follows from (\ref{ops_orthog}) that
\begin{equation*}
  \int_{\mathfrak{G}} \ddoAW x\; w(x)[p_n(x)]^2 = c_{n,n}(a)h_n(a),
  \quad n \geq 0 ,
\end{equation*}
and thus for $ p_n(x) $ to be normalised as well as orthogonal we set $ c_{n,n}(a)h_n(a)=1 $. We also have
moment determinant representations of the polynomials
\begin{equation*}
   p_{n}(x) = \frac{c_{n,n}(b)}{\Delta_{n}(a,b)}
               \det\left(
               \begin{array}{cccccc}
               m_{0,0}(a,b)   & \cdots & m_{0,j}(a,b)   & \cdots & m_{0,n}(a,b)   \\
               \vdots    & \vdots & \vdots    & \cdots & \vdots    \\
               m_{n-1,0}(a,b) & \cdots & m_{n-1,j}(a,b) & \cdots & m_{n-1,n}(a,b) \\
               \phi_{0}(x;b)     & \cdots & \phi_{j}(x;b)     & \cdots & \phi_{n}(x;b)     \\
               \end{array} \right) , \quad n\geq 0.
   \label{ops_polyDet}
\end{equation*}
The normalisation and the three-term recurrence coefficients are related to these determinants by
\begin{alignat*}{2}
   \gamma_n^2 & =   (4ab)^{n} q^{n(n-1)}    \frac{\Delta_{n}(a,b)}{\Delta_{n+1}(a,b)}, & \quad n & \geq 0 ,
%   \label{ops_gammaDelta}
\\
   a^2_n & = \frac{1}{4ab q^{2n-2} } \frac{\Delta_{n+1}(a,b)\Delta_{n-1}(a,b)}{\Delta^2_{n}(a,b)} , & \quad n &\geq 1 ,
   \label{ops_aSQDelta}\\
   b_n & = \tfrac{1}{2} (a q^{n} + a^{-1}q^{-n} ) - \frac{1}{2 a q^{n} } \frac{\Sigma_{n+1,n}(a,b)}{\Delta_{n+1}(a,b)}
 + \frac{1}{2 a q^{n-1} }  \frac{\Sigma_{n,n-1}(a,b)}{\Delta_n(a,b)} , & \quad n & \geq 0 . 
%   \label{ops_bDelta}
\end{alignat*}
Naturally $ p_{n} $, the normalisation and three-term recurrence coefficients are independent of the choices of 
$ a, b $ as can be easily verified from their determinantal definitions given previously.

We also need the definition of the {\it Stieltjes function}
\begin{equation*}
   f(x) \equiv \int_{\mathfrak{G}} \ddoAW y\;\frac{w(y)}{x-y} , \quad x \notin \mathfrak{G} ,
\label{ops_stieltjes}
\end{equation*}
which is a moment generating function in the following sense - it has the following formal expansion as 
$ x \to \infty $ with $ x \neq x_k(b) $, for $ k\in 2\mathbb{Z} $ (i.e. $ \phi_{k}(x;b)\neq 0 $, $ \phi_{\infty}(x;b)\neq 0 $)
and $ b \neq 0, \infty $
\begin{gather}
  f(x) = \frac{f_{\infty}(x;b)}{\phi_{\infty}(x;b)}-2b\sum^{\infty}_{n=0}\frac{q^n}{\phi_{n+1}(x;b)}m_{0,n}(.,b) , \quad x \not \in \mathfrak{G},
\label{xLarge_SF:a} \\
\noalign{\hbox{where}}
  f_{\infty}(x;b) = \int_{\mathfrak{G}} \ddoAW y\;w(y)\frac{\phi_{\infty}(y;b)}{x-y} .
\label{xLarge_SF:b}
\end{gather}
Thus the Stieltjes function splits into two parts - one part being a series with inverse basis polynomials and
a remainder which may be absent for some cases. The Stieltjes function can also be generalised thus
\begin{equation*}
    f_{j}(x;a) := \int_{\mathfrak{G}} \ddoAW y\; w(y)\frac{\phi_j(y;a)}{x-y}, \quad j \geq 1 ,\quad f_0=\phi_0 f .
\label{ops_}
\end{equation*}
Likewise this function has a generating function expansion analogous to (\ref{xLarge_SF:a})
\begin{equation*}
   f_j(x;a) = \frac{f_{j,\infty}(x;a,b)}{\phi_{\infty}(x;b)} - 2b \sum^{\infty}_{n=0} \frac{q^{n}}{\phi_{n+1}(x;b)}m_{j,n}(a,b),
   \quad x \notin \mathfrak{G},\quad x \to \infty .
\label{Sfj_expand}
\end{equation*}

The first crucial structure in our description is the notion of the {\it $\ddoAW$-semi-classical weight},
as given by the following definition of Magnus \cite{Ma_1988}.
\begin{definition}[\cite{Ma_1988}]\label{spectral_DD}
The {\it $\ddoAW$-semi-classical weight} satisfies a first order divided-difference equation
\begin{equation}
    W\ddoAW_{x} w = 2V\moAW_{x} w ,
\label{spectral_DD_wgt:a}
\end{equation}
or equivalently
\begin{equation}
    \frac{w(y_+)}{w(y_-)} = \frac{W+\Delta yV}{W-\Delta yV}(x) ,
\label{spectral_DD_wgt:b}
\end{equation}
with $ W(x),V(x) $ being irreducible polynomials in $ x $, which we will call {\it spectral data polynomials}.
Furthermore we assume $ W\pm \Delta yV \neq 0 $ for all $ x \in \mathfrak{G} $. For minimal degrees of $ W, V $ this 
is the analogue of the Pearson equation. We shall see that the minimal degrees in the generic case are
$ {\rm deg}_{x}W = 2 > {\rm deg}_{x}V = 1 $, i.e. the $ N=2 $ Askey-Wilson situation.
\end{definition}

\begin{remark}
In a series of works Suslov and collaborators \cite{AS_1988,Su_1989,AS_1990,RS_1994a,RS_1994b} have studied
the discrete analogs of the Pearson equation for all of the lattices admissible in the classification and
sought solutions for the weight functions given suitable polynomials for $ W, V $. In terms of our own variables
those found in \cite{Su_1989} are given by
\begin{equation*}
  \sigma = E^{-}_x(W-\Delta yV), \qquad
  \tau = \frac{E^{+}_x(W+\Delta yV)(x)-E^{-}_x(W-\Delta yV)(x)}{\Delta y} .
\end{equation*}
\end{remark}

From (\ref{xLarge_SF:a}) we recognise $ f(x) $ as a moment generating function and a key element in our theory
are the systems of divided-difference equations satisfied by the moments. Before we state such systems we
need to note the following preliminary result.
\begin{lemma}[\cite{Su_1989},\cite{Wi_2010a}]
Let $ \phi_k(x;a) $ be a canonical basis polynomial. For any $ k\in\mathbb{Z}_{\geq 0} $ and 
$ a\in\mathbb{C}\backslash\{0,\infty\} $ the $\ddoAW$-semi-classical weight satisfies the integral equation
\begin{equation}
   \int_{\mathfrak{G}}\ddoAW{x}\,\frac{1}{\Delta y}
   E^{-}_x\Big( w(x) \left[ E^{+}_x \phi_k(x;a)(W+\Delta yV)(x)-E^{-}_x \phi_k(x;a)(W-\Delta yV)(x) \right] \Big)
   = 0 . 
\label{SCwgtEqn}
\end{equation}
\end{lemma}

An equivalent equation for general quadratic lattices has been derived by Suslov \cite{Su_1989} in the special
case of $ k=0 $ of (\ref{SCwgtEqn}).
From the above result we see the need to express $ W\pm\Delta yV $ in terms of canonical basis polynomials 
and therefore make two definitions in the following way
\begin{gather} 
E^{+}_x(W+\Delta yV)(x)+E^{-}_x(W-\Delta yV)(x) =: 2\sum^{N}_{l=0}f_{N,l}(a)\phi_l(x;a) =: H(x) ,
\label{Sum_Bexp}
  \\
E^{+}_x(W+\Delta yV)(x)-E^{-}_x(W-\Delta yV)(x) =: (z-z^{-1}) \sum^{N-1}_{l=0}g_{N,l}(a)\phi_l(x;a) =: (z-z^{-1})G(x) .
\label{Diff_Bexp}
\end{gather} 
Here the cutoff in the sums are determined by $ {\rm deg}_{x}W = N $, $ {\rm deg}_{x}V = N-1 $.

\begin{corollary}[\cite{Wi_2010a}]\label{momentRecur}
Let us assume that $ a_j \neq a_i $. For all $ k,l \in\mathbb{Z}_{\geq 0} $ the moments $ m_{k,l}(a_j,a_i) $ 
are characterised by the linear recurrence relations
\begin{multline}
   \tfrac{1}{2}(1+q^{-k})\sum^{N-1}_{l=0} g_{N,l}(a_i)m_{k,l}(a_j,a_i)
  \\
  +q^{-1}(1-q^k)a_j\sum^{N}_{l=0} \left[ f_{N,l}(a_i)+\tfrac{1}{2}(q^{k-1}a_j-q^{-k+1}a_j^{-1})g_{N,l}(a_i) \right]
                 m_{k-1,l}(a_j,a_i)
  = 0 .
\label{Mrecur}
\end{multline}
For $ k=0 $ the second term is absent. In the generic case $ N-1 $ initial values of the $ m_{k,l} $ need to
be specified.
\end{corollary}
\begin{proof}
This follows from a specialisation of Corollary 4.1 (which itself is a direct consequence of \eqref{SCwgtEqn}) on p.~15
of \cite{Wi_2010a} to the $q$-quadratic lattice.
\end{proof}

The Stieltjes function itself satisfies a linear first-order divided-difference equation and this 
also characterises the $\ddoAW$-semiclassical system.
\begin{proposition}[\cite{Ma_1988},\cite{Wi_2010a}]\label{prop:diff_eqn}
Given Definition \ref{spectral_DD} and the conditions therein then the Stieltjes function satisfies the 
inhomogeneous form of \eqref{spectral_DD_wgt:a}
\begin{equation}
    W\ddoAW_{x} f = 2V\moAW_{x} f + U .
\label{spectral_DD_st}
\end{equation}
Of particular relevance to our application are the {\it $\ddoAW$-semi-classical class} of orthogonal
polynomial systems defined by the property that $ U(x), V(x) $ and $ W(x) $
in (\ref{spectral_DD_st}) are polynomials in $ x $.
\end{proposition}
Proposition~\ref{prop:diff_eqn} also furnishes alternative means to derive linear divided-difference equations 
for the moments $ m_{0,n} $, given that $ U(x) $ is a polynomial of fixed (and known) degree.
In addition knowledge of the moments is sufficient to determine the polynomial $ U(x) $,
and this is given explicitly by the following result.
\begin{proposition}[\cite{Wi_2010a}]\label{Ueval} 
For arbitrary $ a \in \mathbb{C} $ the polynomial $ U(x) $ is 
\begin{multline}
  (q^{1/2}-q^{-1/2})U(x) = -4a^2\sum^{N}_{k=2}f_{N,k}(a)\sum^{k-2}_{n=0}q^{n-1/2}(q^{n+1}-q^{k})m_{0,n}(a)\phi_{k-n-2}(x;q^{n+3/2}a)
\\
      +2a\sum^{N-1}_{k=1}g_{N,k}(a)\sum^{k-1}_{n=0}q^{n-k+1/2}(q^{k}+q^{n})m_{0,n}(a)\phi_{k-n-1}(x;q^{n+3/2}a)
\\
      -2a\sum^{N-1}_{k=2}g_{N,k}(a)\sum^{k-2}_{n=0}q^{n-k-1/2}(1-a^2q^{2k})(q^{n+1}-q^{k})m_{0,n}(a)\phi_{k-n-2}(x;q^{n+3/2}a) ,
\label{Upoly} 
\end{multline}
which despite appearances is independent of $ a $. Note that this expression only has meaning if $ N \geq 2 $.
\end{proposition} 
\begin{proof}
We start with \eqref{spectral_DD_st} in the form
\begin{equation}
   \Delta yU(x) = (W-\Delta yV)f(y_{+})-(W+\Delta yV)f(y_{-}) ,
\label{aux:Z}
\end{equation}
and employ the expansions (\ref{Sum_Bexp},\ref{Diff_Bexp}) for $ E_x^{\pm}(W\pm\Delta yV) $ and
(\ref{xLarge_SF:a}) for the Stieltjes function, with their base parameters equal at $ a $. In the summand of 
the terms arising from $ E_{x}^{+}Hf-E_{x}^{-}Hf $ we resolve the ratio of two basis polynomials as
\begin{equation*}
    \frac{\phi_{k}(x;a)}{\phi_{n+1}(x;a)} = 
    \begin{cases}
        \phi_{k-n-1}(x;q^{n+1}a) & k > n+1 \\
        1                        & k = n+1 \\
        \frac{\displaystyle 1}{\displaystyle \phi_{n-k+1}(x,q^{k}a)} & k < n+1 
    \end{cases} .
\end{equation*}
The action of the divided-difference operator on the above expressions is computed using (\ref{DDO_basis:a})
and the identity
\begin{equation*}
   \phi_{n-k+1}(y_{+},q^ka)\phi_{n-k+1}(y_{-},q^ka) = \phi_{n-k+1}(x,q^{k+1/2}a)\phi_{n-k+1}(x,q^{k-1/2}a) .
\end{equation*}
In the summand of the terms arising from $ (q^{1/2}z-q^{-1/2}z^{-1})E_{x}^{+}Gf+(q^{-1/2}z-q^{1/2}z^{-1})E_{x}^{-}Gf $
we find the cases according to $ k\geq n+1 $
\begin{multline*}
   (q^{1/2}z-q^{-1/2}z^{-1})\phi_{k-n-1}(y_{+};q^{n+1}a)+(q^{-1/2}z-q^{1/2}z^{-1})\phi_{k-n-1}(y_{-};q^{n+1}a)
\\
   = (z-z^{-1})\left[ q^{-k+1/2}(q^k+q^n)\phi_{k-n-1}(x,q^{n+3/2}a)-q^{-k-1/2}(1-a^2q^{2k})(q^{n+1}-q^k)\phi_{k-n-2}(x;q^{n+3/2}a) \right] ,
\end{multline*}
and $ k \leq n $
\begin{multline*}
   \frac{q^{1/2}z-q^{-1/2}z^{-1}}{\phi_{n-k+1}(y_{+};q^{k}a)}+\frac{q^{-1/2}z-q^{1/2}z^{-1}}{\phi_{n-k+1}(y_{-};q^{k}a)}
\\
   = (z-z^{-1})\left[ \frac{q^{-n-1/2}(q^k+q^n)}{\phi_{n-k+1}(x,q^{k-/2}a)}+\frac{q^{-n-1/2}(1-a^2q^{2n})(q^{n+1}-q^k)}{\phi_{n-k+2}(x;q^{k-1/2}a)} \right] .
\end{multline*}
Combining these two contributions and only retaining the polynomial part we get (\ref{Upoly}).
\end{proof}

\section{Linear \texorpdfstring{$q$}{q}-difference Equations for \texorpdfstring{$I_{N}$}{IN}}\label{differenceIN}
\setcounter{equation}{0}
\subsection{Moment Problem on \texorpdfstring{$q$}{q}-quadratic Lattice}

In this subsection we apply the theory of the previous section to the specific weight appearing in 
\eqref{Witte_integral} and thus derive precise consequences from Corollary \ref{momentRecur}. If a
specialisation of the parameters $ a, b $ appearing in the basis functions in \eqref{ops_moment} is made
to effect some cancellation with corresponding factors in the weight for \eqref{Witte_integral} then any moment
is given by an integral whose integrand has the same structure. Such integrals can in fact be written as a
certain $ I_{N} $. For example if $ a \mapsto a_r $, $ b \mapsto a_s $ with $ s>r $ then
$ m_{j,k}(a_r,a_s) = I_N(\ldots,q^ja_r,\ldots,q^ka_s,\ldots) $ for any $ j,k \in \mathbb{Z}_{\geq 0} $.
Consequently \eqref{Mrecur} becomes a linear divided-difference equation for $ I_{N} $ with
respect to the internal parameters of the weight. 

Let us denote the $k$-th elementary symmetric function of $ \{q^{-1/2}a_1,\ldots,q^{-1/2}a_{2N}\} $ by $ \tilde{\sigma}_{k} $
whereas the $j$-th elementary symmetric polynomial of $ \{a_1,\ldots,a_{2N}\} $ will be denoted by $ \sigma_j $. 
The weight data for \eqref{Witte_integral} is computed to be
\begin{equation*}
  W\pm \Delta y V = z^{\mp N}\prod^{2N}_{j=1}(1-a_jq^{-1/2}z^{\pm 1}) ,
\label{deform_AW_SCff:z}
\end{equation*}
which implies, as found by \cite{ILR_2004}, that
\begin{equation*}
   W(x) = (-1)^N\tilde{\sigma}_N+\sum^{N-1}_{l=0}(-1)^l(\tilde{\sigma}_{l}+\tilde{\sigma}_{2N-l})T_{N-l}(x) ,
\end{equation*}
and
\begin{equation*}
   V(x) = -\frac{1}{q^{1/2}-q^{-1/2}}\sum^{N-1}_{l=0}(-1)^l(\tilde{\sigma}_{l}-\tilde{\sigma}_{2N-l})U_{N-l-1}(x) ,
\end{equation*}
where $ T_k(x), U_k(x) $ are the $k$-th Chebyshev polynomials of the first and second kinds 
respectively. A computation using the expansion theorem of Ismail \cite{{Is}_1995} and the 
formula of Cooper \cite{IS_2003} allows us to deduce that the expansion coefficients \eqref{Sum_Bexp}, \eqref{Diff_Bexp}
are given by
\begin{equation}
    2f_{N,k}(a) = \frac{q^{k-N/2}}{(q;q)_k}\sum^{2N}_{m=0}(-1)^m(\sigma_{m}+\sigma_{2N-m})
                      \sum^{k}_{l=0} \begin{bmatrix} k \\ l \end{bmatrix}_{q}
                                     \frac{a^{2(l-k)+N-m}q^{-(k-l)^2+(k-l)(N-m)}}
                                          {(q^{1+2k-2l}a^2;q)_{l}(q^{1+2l-2k}a^{-2};q)_{k-l}} ,
\label{fCoeff_N}
\end{equation}
and
\begin{equation}
     g_{N,k}(a) = -\frac{q^{k-N/2}}{(q;q)_k}\sum^{2N}_{m=0}(-1)^m(\sigma_{m}-\sigma_{2N-m})a^{m-N}
                      \sum^{k}_{l=0} \begin{bmatrix} k \\ l \end{bmatrix}_{q}
                                     \frac{a^{1+2(l-k)}q^{-(k-l)^2+(k-l)(m+1-N)}}
                                          {(q^{2k-2l}a^2;q)_{l+1}(q^{1+2l-2k}a^{-2};q)_{k-l}} .
\label{gCoeff_N}
\end{equation}

Setting $ k=0 $ in Corollary \ref{momentRecur} immediately implies the following result. 
\begin{proposition}\label{prop:q_diff_eqn}
The integrals $ I_{N} $ satisfy a $(N-1)$-th order linear $q$-difference equation with respect to $ a_i $,
for any $ i\in \{1,\ldots,2N\} $
\begin{equation}
    \sum^{N-1}_{l=0} g_{N,l}(a_i)I_N(\ldots,q^l a_i,\ldots) = 0 ,
\label{I_q-DifferenceEqn}
\end{equation}
where $ g_{N,l} $ is given by (\ref{gCoeff_N}).
\end{proposition}

\begin{remark}
The integrals $ I_{N} $ satisfy the general identity
\begin{equation}
   a_k I_{N}(\ldots,qa_j,\ldots,a_k,\ldots) -a_j I_{N}(\ldots,a_j,\ldots,qa_k,\ldots) = (a_k-a_j)(1-a_ja_k) I_{N}(\ldots,a_j,\ldots,a_k,\ldots) ,
\label{Iidentity}
\end{equation}
which applies to any integral with products of $ \phi_{\infty}(x;a_j)\phi_{\infty}(x;a_k) $ in the denominator
of the integrand and any distinct pair $ a_j \neq a_k $.
\end{remark}

It is instructive at this point to see what the above general result implies in the first $ N=2 $ case even 
though virtually all the results are well known. From the Askey-Wilson weight (\ref{AWwgt}) we compute the
spectral data to be
\begin{equation}
   W\pm \Delta yV = z^{\mp 2}\prod^{4}_{j=1}(1-a_j q^{-1/2}z^{\pm 1}) ,
\label{N=2Sdata}
\end{equation}
from which we deduce
\begin{gather*}
  W(x) = 2(1+\sigma_4 q^{-2})x^2 - \left[ q^{-1/2}\sigma_1+q^{-3/2}\sigma_3 \right]x
           -1+q^{-1}\sigma_2-q^{-2}\sigma_4 ,
\label{AW_polyW} \\
 V(x) = 2\frac{q^{-2}\sigma_4-1}{q^{1/2}-q^{-1/2}}x
     +\frac{q^{-1/2}\sigma_1-q^{-3/2}\sigma_3}{q^{1/2}-q^{-1/2}} .
\label{AW_polyV}
\end{gather*}
We only require the $ k=0 $ case of (\ref{Mrecur}) and from (\ref{N=2Sdata}) we compute the 
relevant coefficients as
\begin{equation}
  g_{2,1}(a) = q^{-1}a^{-1}(1-\sigma_4), \qquad g_{2,0}(a) = q^{-1}(a+a^{-1})(\sigma_4-1)+q^{-1}(\sigma_1-\sigma_3) .
\end{equation}

We find that the general system of moment recurrences, specialised to the case $ N=2 $, coincides with the
recurrence of Kalnins and Miller\cite{KM_1989}, and Koelink and Koornwinder\cite{KK_1992}.
\begin{proposition}[\cite{KM_1989},\cite{KK_1992}]
The Askey-Wilson integral, using the case of $ a=a_1 $ without loss of generality, satisfies the two-term
linear recurrence
\begin{equation}
   (\sigma_4-1)I_{2}(qa_1,a_2,a_3,a_4) = (a_1a_2-1)(a_1a_3-1)(a_1a_4-1)I_{2}(a_1,a_2,a_3,a_4) , 
\label{AWintegral}
\end{equation}
which is solved by
\begin{equation*}
  I_{2}(a_1,a_2,a_3,a_4) 
   = 2\frac{(\sigma_4;q)_{\infty}}{(q;q)_{\infty}\prod_{k>j}(a_{j}a_{k};q)_{\infty}} .
\end{equation*}
Consequently the moments are given by
\begin{equation*}
   m_{0,n}(a_1)
   = \pi\frac{(a_1a_2,a_1a_3,a_1a_4;q)_{n}}{(\sigma_4;q)_{n}}I_{2}(a_1,a_2,a_3,a_4)
   = 2\pi\frac{(q^{n}\sigma_4;q)_{\infty}}
                {(q^{n}a_1a_2,q^{n}a_1a_3,q^{n}a_1a_4,a_2a_3,a_2a_4,a_3a_4,q;q)_{\infty}} .
\label{AWmoment}
\end{equation*}
\end{proposition}

Upon an application of the general formula for $ U $ derived in Proposition \ref{Ueval}, using the coefficient
$ f_{2,2} = (1+\sigma_4)/2q^2a^2 $ and the evaluation of the moment $ m_{0,0} $ we deduce \cite{Wi_2010a}
\begin{equation*}
   U =\frac{8\pi}{q-1}\frac{(q^{-1}\sigma_4;q)_{\infty}}{(q;q)_{\infty}\prod_{k>j}(a_ja_k;q)_{\infty}} . 
\end{equation*}

Now we turn our attention to the $ N=3 $ and $ L=1 $ case. The explicit form for the deformed weight leading
to spectral polynomials with degrees $ 2N=6 $ \eqref{N=3wgt} is 
\begin{equation*} 
  w(x;u) = w(x;\{a_1,\ldots,a_6\}) 
  = \frac{(z^{\pm 3};q^{3/2})_{\infty}}{\sin(\frac{3}{2}\theta)\prod^{6}_{j=1}(a_jz^{\pm 1};q)_{\infty}} .
\label{deform_AWwgt}
\end{equation*}
With the parameterisation $ a_5=\alpha t $, $ a_6=\alpha t^{-1} $ and the four ``fixed'' parameters $ a_1,\ldots,a_4 $ 
appearing in the same form as they do in the Askey-Wilson weight ($ N=2 $) we have introduced a deformation
parameter $ \alpha $ and its associated deformation variable, $ u=\tfrac{1}{2}(t+t^{-1}) $. The additional
factors in the weight achieve a $BC_n$ type structure of the form $ (\alpha tz,\alpha tz^{-1},\alpha t^{-1}z,\alpha t^{-1}z^{-1};q)_{\infty} $.
Thus $ W\pm \Delta yV $ have the simple form
\begin{equation}
  W\pm \Delta y V = z^{\mp 3}\prod^{6}_{j=1}(1-a_jq^{-1/2}z^{\pm 1}) .
\label{deform_AW_SCff:a}
\end{equation}
As a consequence of (\ref{deform_AW_SCff:a}) we have
\begin{gather*}
  W(x) = 4(1+\tilde{\sigma}_6)x^3-2(\tilde{\sigma}_1+\tilde{\sigma}_5)x^2
        +(\tilde{\sigma}_2+\tilde{\sigma}_4-3-3\tilde{\sigma}_6)x
        +\tilde{\sigma}_1-\tilde{\sigma}_3+\tilde{\sigma}_5 ,
\label{deform_AW_SCff:b} \\
  V(x) = \frac{1}{q^{1/2}-q^{-1/2}}
        \left[ -4(1-\tilde{\sigma}_6)x^2
        +2(\tilde{\sigma}_1-\tilde{\sigma}_5)x+1-\tilde{\sigma}_2+\tilde{\sigma}_4-\tilde{\sigma}_6
        \right] ,
\label{deform_AW_SCff:c}
\end{gather*}
Furthermore the expansion coefficients have the evaluations
\begin{align*}
   g_{3,2}(a) & = \frac{\sigma_6-1}{q^{5/2}a^2} ,
  \\
   g_{3,1}(a) & = \frac{1}{q^{5/2}a^2}[qa(\sigma_5-\sigma_1)-(1+q)(1+qa^2)(\sigma_6-1)] ,
  \\
   g_{3,0}(a) & = \frac{1}{q^{3/2}a^2}[(1+a^2+a^4)(\sigma_6-1)-(a+a^3)(\sigma_5-\sigma_1)+a^2(\sigma_4-\sigma_2)]
                = \frac{1}{q^{3/2}a^2} \prod_{a_{j} \neq a}(a a_{j}-1) .
\end{align*}

A generalisation of the two-term recurrence for the Askey-Wilson integral given in (\ref{AWintegral}) is the 
following result.
\begin{proposition}[\cite{Wi_2010a}]\label{N=3recurrence}
The integral $ I_3 $ satisfies a three-term recurrence in a single variable, which we take without loss of
generality to be $ a_1 $,
\begin{multline}
   0 = \prod_{j\neq 1}(a_1a_j-1)I_{3}(a_1,\ldots)  \\ 
	 +\left[1+q^{-1}-a_1 \left(\sum_{j\neq 1}a_j-qa_1 \right)
   	+\prod^6_{1}a_j \left(a_1\sum^5_{k\neq 1}a_{k}^{-1}-q^{-1}-(q+1)a_1^2 \right) \right] I_{3}(qa_1,\ldots)
   \\  +q^{-1} \left(\prod^6_1 a_j-1 \right)   I_{3}(q^2a_1,\ldots ) .
\label{N=3int_Recur:a}
\end{multline}
Let $ \sigma_k $ denote the $k$-th elementary symmetric polynomial in $ a_1,\ldots,a_4 $.
In addition the integral $ I_{3} $ satisfies a three-term recurrence in two variables, taken to be
with respect to $ a_5, a_6 $, which constitutes a pure recurrence in the deformation variable $ u $
\begin{multline} 
   0 = (a_5-qa_6)\prod^{4}_{j=1}(1-a_ja_6) I_{3}(\ldots,q^2a_5,a_6)
   \\ -(a_5-a_6)\Big[ (1+q)(1+qa_5a_6\sigma_2+q^2a_5^2a_6^2\sigma_4)-(qa_5-a_6)(qa_6-a_5)(q+\sigma_4)
   \\                 -q(a_5+a_6)(\sigma_1+qa_5a_6\sigma_3) \Big] I_{3}(\ldots,qa_5,qa_6)
   \\ +(qa_5-a_6)\prod^{4}_{j=1}(1-a_ja_5) I_{3}(\ldots,a_5,q^2a_6) .
\label{N=3int_Recur:b}
\end{multline} 
\end{proposition}
\begin{proof}
Equation \eqref{N=3int_Recur:a} follows directly from \eqref{Mrecur}. The second result follows from the first
rewritten as recurrences in $ a_5 $ and $ a_6 $ and the repeated use of the identity \eqref{Iidentity}.
\end{proof}

\begin{remark}
For convenience let us define $ w_0(z) = \prod^{4}_{j=1}(1-q^{-1/2}a_j z) $. Then (\ref{N=3int_Recur:b}) can be
rewritten in the form of the three-term recurrence for the moment $ m_{0,0} $ with respect to $ t $
\begin{multline}
   (q^{-1/2}t-q^{1/2}t^{-1})w_0(q^{-1/2}\alpha t^{-1})m_{0,0}(qt)
  \\
   -q^{-1/2}(t-t^{-1})
    \left[ (1+q)(1+q^{-1}\alpha^2\sigma_{2}+q^{-2}\alpha^4\sigma_{4})
               +q^{-1}\alpha^2(q+\sigma_{4})(q^{1/2}t-q^{-1/2}t^{-1})(q^{-1/2}t-q^{1/2}t^{-1})
    \right.
  \\
    \left.
                   -\alpha(t+t^{-1})(\sigma_{1}+q^{-1}\alpha^2\sigma_{3})
    \right]m_{0,0}(t)
  \\
   +(q^{1/2}t-q^{-1/2}t^{-1})w_0(q^{-1/2}\alpha t)m_{0,0}(q^{-1}t) = 0 .
\label{3TermMoment}
\end{multline}
\end{remark}

\begin{proposition}
Assume that $ |\alpha^2| > q $ and $ q^{-1}\sigma_4\alpha^2, a_j\alpha t \notin q^{-\mathbb{Z}_{\geq 0}} $.
The two independent solutions of (\ref{3TermMoment}) are
\begin{multline}
  m_{0,0}(t) = t^{1/2}(q\alpha t,\alpha^{-1}t^{-1},q^{1/2}\alpha t,q^{1/2}\alpha^{-1}t^{-1};q)_{\infty}
  \\ \times
         \frac{(a_1^{-1}\sigma_4\alpha t,a_2^{-1}\sigma_4\alpha t,a_3^{-1}\sigma_4\alpha t,a_4^{-1}\sigma_4\alpha t;q)_{\infty}}
              {(a_1\alpha t,a_2\alpha t,a_3\alpha t,a_4\alpha t,t^{-2},\sigma_4\alpha^2t^2;q)_{\infty}}
  \\ \times
   {}_{8}W_{7}(q^{-1}\sigma_4\alpha^2t^2;q^{-1}\sigma_4\alpha^2,a_1\alpha t,a_2\alpha t,a_3\alpha t,a_4\alpha t;q\alpha^{-2}) ,
\label{3_Solution}
\end{multline}
and (\ref{3_Solution}) with $ t \mapsto t^{-1} $.
\end{proposition}
\begin{proof}
We offer a proof by way of verification. 
Let us denote the $ {}_{8}W_{7} $ function in (\ref{3_Solution}) by $ W(t) $. Substituting this into the left-hand
side of (\ref{3TermMoment}) we find an expression proportional to 
\begin{multline*}
   \frac{\prod^{4}_{j=1}(1-q^{-1}\alpha t^{-1}a_j)}{\alpha^2(1-t^{-2})(1-q^{-1}t^{-2})(1-q^{-2}t^{-2})}W(qt)
  \\
   -\frac{1}{q^2(1-t^2)(1-qt^{-2})(1-\sigma_4\alpha^2t^2)(1-q\sigma_4\alpha^2t^2)}
    \frac{\prod^{4}_{j=1}(1-\sigma_4\alpha ta^{-1}_j)}{\prod^{4}_{j=1}(1-\alpha ta_j)}
  \\ \times
    \left[ (q-t^2)\prod^{4}_{j=1}(qt-\alpha a_j)+(1-qt^2)\prod^{4}_{j=1}(q-\alpha ta_j)
          -(1-t^2)(t^2-q)(1-qt^2)(\alpha^2-q)(q^2-\alpha^2\sigma_4)
    \right]W(t)
  \\
   +\frac{\alpha^2t^2(qt^2-1)\prod^{4}_{j=1}(1-q^{-1}\sigma_4\alpha ta^{-1}_j)\prod^{4}_{j=1}(1-\sigma_4\alpha ta^{-1}_j)}
         {(1-q^{-2}\sigma_4\alpha^2t^2)(1-q^{-1}\sigma_4\alpha^2t^2)(1-\sigma_4\alpha^2t^2)(1-q\sigma_4\alpha^2t^2)\prod^{4}_{j=1}(1-\alpha ta_j)}
     W(q^{-1}t) .
\end{multline*}
However this is precisely the contiguous relation of the $ {}_{8}W_{7} $, see Eq. (2.9), p. 278 of \cite{GM_1994} with
$ a=q^{-1}\sigma_4\alpha^2t^2 $, $ b=q^{-1}\sigma_4\alpha^2 $ and $ c,d,e,f = a_{1,2,3,4}\alpha t $.
\end{proof}

Using the formula given in Proposition \ref{Ueval} and that
\begin{align*}
    f_{3,3}(a) & = -\frac{1}{2q^{9/2}a^3}(\sigma_6+1) ,
  \\
    f_{3,2}(a) & =  \frac{1}{2q^{9/2}a^3}\left[ (1+q+q^2)(1+q^2a^2)(\sigma_6+1)-q^2a(\sigma_5+\sigma_1) \right] ,
\end{align*}
we can compute the polynomial $ U $.
\begin{proposition}
Let $ \sigma_k $ denote the $k$-th elementary symmetric polynomial in $ a_1,\ldots,a_6 $. For $ N=3 $ the
polynomial $ U(x) $ is given by
\begin{equation*}
   U(x) = \frac{4}{q^2(q^{1/2}-q^{-1/2})}\left[ m_{0,0}\left( -2q^{-1/2}(\sigma_6-q^2)x+\sigma_5-q\sigma_1 \right)-(\sigma_6-q)(m_{0,+}+m_{0,-}) \right] ,
\label{M=3U}
\end{equation*}
where $ m_{0,\pm} = \int_{\mathfrak{G}} \ddoAW x\,w(x;u)z^{\pm 1} $.
\end{proposition}
\begin{proof}
This is deduced from \eqref{Upoly} appropriately specialised and the following Lemma.
\end{proof}
\begin{lemma}
The moments $ m_{0,\pm } $ and $ m_{0,0} $ satisfy the relation
\begin{multline*}
   (q^{-2}\sigma_6-q^{-1})(m_{0,+}+m_{0,-})(t)
   = \left[ q^{-2}\alpha^2\sigma_3-q^{-1}\sigma_1+q^{-3}(q-\alpha^2)(q^2\alpha^{-1}t^{-1}+\sigma_4\alpha t) \right]m_{0,0}(t)
\\
     -\frac{w_0(q^{-1/2}\alpha t)}{\alpha t}\frac{q^{-1/2}tm_{0,0}(q^{-1}t)-q^{1/2}t^{-1}m_{0,0}(t)}{q^{-1/2}t-q^{1/2}t^{-1}} ,
\end{multline*}
which is essentially a contiguous relation for the $ {}_{8}W_{7} $ solutions.
\end{lemma}

\subsection{First Principles Approach}
We now offer an alternative derivation for the recurrence relation \eqref{N=3int_Recur:a} of $I_N$ from first
principles and which will enable us to derive further linear recurrence relations for these integrals. This
approach proceeds directly from the contour integral definition of $ I_{N} $ and does not draw upon its setting 
as a moment of a semi-classical weight with support on a discrete, quadratic lattice.

Let us define the function $ F(z) $ as 
\begin{equation*}
 F(z):=F_{-}(z)+F_{+}(z) ,
\end{equation*}
where
\begin{equation*}
  F_{-}(z) =\frac{z^N\prod_{i=1}^{2N}(1-a_iz^{-1})}{z-z^{-1}} ,
\quad \mbox{and} \quad
  F_{+}(z)=\frac{z^{-N}\prod_{i=1}^{2N}(1-a_iz)}{z^{-1}-z} .
\end{equation*} 
By definition $F(z)$ is symmetric with respect to the reflection $z\to z^{-1}$, allowing its expression as 
\begin{equation}
F(z)=\frac{z}{1-z^2}\sum_{m=0}^{2N}(-1)^m (\sigma_m-\sigma_{2N-m}) z^{m-N},
\label{eq:F01}
\end{equation}
where $\sigma_i$ is the $i$-th elementary symmetric function of $a_1,\ldots,a_{2N}$, which will be used later.

\begin{remark} 
Through Lemmas \ref{lem:G} and \ref{lem:C} to follow we will give two types of expansions for $F(z)$, each of which provides a 
$q$-difference equation. The following fundamental lemma shows that these two $q$-difference equations are inter-convertible. 
\end{remark}

\begin{lemma}[Fundamental relation]\label{lem:FR}
For an holomorphic function $\varphi(z)$ on $\mathbb{C}^*$ the following holds:
\begin{equation}
  \int_{\mathbb{T}}\Big(\varphi(qz)F_{+}(z)+\varphi(z)F_{-}(z)\Big)\Phi(z)\frac{dz}{z}=0.
\label{eq:FR1}
\end{equation}
In particular for $\varphi(z)=1$ 
\begin{equation}
  \int_{\mathbb{T}}F(z)\Phi(z)\frac{dz}{z}=0,
\label{eq:FR2}
\end{equation}
where the contour $\mathbb{T}$ is chosen appropriately for $N$ odd.
\end{lemma}
\begin{proof}
Since $F_{+}(z)\Phi(z)=-F_{-}(qz)\Phi(qz)$, in order to prove (\ref{eq:FR1}) it suffices to show that 
\begin{equation}
  \int_{\mathbb{T}}\varphi(qz)F_{-}(qz)\Phi(qz)\frac{dz}{z} = \int_{\mathbb{T}}\varphi(z)F_{-}(z)\Phi(z)\frac{dz}{z}.
\label{eq:FR3}
\end{equation}
By the change of variables $z'=qz$, the left-hand side of (\ref{eq:FR3}) is written as
\begin{equation*}
  \int_{q\mathbb{T}}\varphi(z)F_{-}(z)\Phi(z)\frac{dz}{z} , 
\end{equation*} 
where $ {\mathbb{T}}=\{z\in \mathbb{C}\,;\,|z|=1\} $, $ q{\mathbb{T}}:=\{z\in \mathbb{C}\,;\,|z|=|q|\} $.
Since the function 
\begin{equation*}
  \varphi(z)F_{-}(z)\Phi(z)=\varphi(z)\frac{z^N\prod_{i=1}^{2N}(1-a_iz^{-1})}{z-z^{-1}}\Phi(z)
  = \varphi(z)\frac{z^{N/2}(z;q^{N/2})_\8(q^{N/2}z^{-1};q^{N/2})_\8}{\prod_{j=1}^{2N}(a_jz;q)_\8(qa_jz^{-1};q)_\8} ,
\end{equation*} 
has no poles in the annulus $\{z\in \mathbb{C}\,;\,|q|\le |z|\le 1\}$ under the assumption $|a_j|<1$,
the contour can be deformed from $q\mathbb{T}$ to $\mathbb{T}$, i.e.
\begin{equation*}
  \int_{q\mathbb{T}}\varphi(z)F_{-}(z)\Phi(z)\frac{dz}{z} = \int_{\mathbb{T}}\varphi(z)F_{-}(z)\Phi(z)\frac{dz}{z} ,
\end{equation*} 
and the result is shown.
\end{proof}

For a non-negative integer $n$ we set $\tilde{\phi}_n(z;a):=(az;q)_n(az^{-1};q)_n$, which satisfies 
\begin{equation}
\tilde{\phi}_n(z;a)=0\quad\mbox{if} \quad z=a,qa,q^2a,\ldots, q^{n-1}a.
\label{eq:vanish}
\end{equation}
Since $\tilde{\phi}_1(z;a)/a=a+a^{-1}-z-z^{-1}$, the trivial three-term relation
\begin{equation}
   \frac{\tilde{\phi}_1(z;a_i)}{a_i} - \frac{\tilde{\phi}_1(z;a_j)}{a_j} = \frac{\tilde{\phi}_1(a_j;a_i)}{a_i} 
\label{eq:T1}
\end{equation}
holds for $\tilde{\phi}_1(z;a)$. From our definition $\tilde{\phi}_n(z;a_1)$ is related to the $q$-shift of the integral
$I_N$ by
\begin{equation}
  I_N(q^n a_1,a_2,\ldots,a_{2N}) = \int \tilde{\phi}_n(z;a_1)\Phi(z)\frac{dz}{2\pi \sqrt{-1}z}. 
\label{eq:T2}
\end{equation}
From (\ref{eq:T1}) we immediately have (see \eqref{Iidentity})
\begin{equation}
  I_N(\ldots,qa_i,\ldots)- a_ia_j^{-1}I_N(\ldots,qa_j,\ldots)=\tilde{\phi}_1(a_j;a_i)I_N(a_1,\ldots,a_{2N}),
\label{eq:T3}
\end{equation}
which plays an important role later when we discuss the first order $q$-difference system of $I_N$ in Corollary~\ref{cor:1st}.

\begin{theorem}[Proposition~\ref{prop:q_diff_eqn}]\label{thm:gI}
$I_{N}$ satisfies the $(N-1)$-th order linear $q$-difference equation
\begin{equation}
    \sum_{i=0}^{N-1}g_i(a_1)I_N(q^ia_1,a_2,\ldots,a_{2N}) = 0 , 
\label{eq:g02-1}    
\end{equation}
where $ g_i(a) $ is given by
\begin{equation}
   g_i(a) = aq^i\sum_{j=0}^i\sum_{m=0}^{2N}(-1)^{j+m}(\sigma_m-\sigma_{2N-m}) 
          \frac{(q^ja)^{m-N}q^{\binom{j+1}{2}}}{(q;q)_{i-j}(q;q)_{j}(q^{2j}a^2;q)_{i-j+1}(q^{j}a^2;q)_{j}} .
\label{gi_new}
\end{equation} 
\end{theorem}

\begin{remark}
Equation (\ref{gi_new}) for $ g_{i}(a) $ exactly coincides with expression \eqref{gCoeff_N} for $g_{N,i}(a)$ 
up to the constant $-q^{-N/2}$ which presents no problems since the constant is multiplied over all coefficients 
$g_{N,i}(a)$, $0\le i\le N-1$.
\end{remark}

According to (\ref{eq:FR2}) of Lemma \ref{lem:FR} and (\ref{eq:T2}), in order to prove Theorem~\ref{thm:gI} it is suffices to show the following. 
\begin{lemma}\label{lem:G}
The function $F(z)$ can be expanded as 
\begin{gather}
  F(z)=\sum_{i=0}^{N-1}g_i\,\tilde{\phi}_i(z;a), 
\label{eq:g01}
\end{gather}
where the constant $ g_i=g_{i}(a) $ is given by \eqref{gi_new}.
\end{lemma}
\begin{proof} 
Since we have the expression (\ref{eq:g01}), and taking account of the degree of $F(z)$ as a Laurent polynomial,  
we can now evaluate $g_i$. From the vanishing property of $\tilde{\phi}_i(z;a)$ in (\ref{eq:vanish}), we have  
\begin{equation*}
   F(q^{i}a)=\sum_{j=0}^{i}g_j\,\tilde{\phi}_j(q^{i}a;a) ,\quad i=0,1,\ldots, N-1 .
\label{eq:sim_eqns}
\end{equation*}
Conversely we can solve the above simultaneous linear equations relating $g_j$ and $F(q^{i}a)$, so that $g_i$ 
can be expressed as 
\begin{equation}
   g_i=\sum_{j=0}^i c_{ij}\,F(q^{j}a) ,
\label{eq:g03}
\end{equation}
where $(c_{ij})_{i,j=0}^{N-1}$ is the inverse of the lower triangular matrix $\big(\tilde{\phi}_j(q^{i}a;a)\big)_{i,j=0}^{N-1}$. 
This is explicitly written as $c_{ij}=0$ if $i<j$ and 
\begin{equation}
  c_{ij} = \frac{(-1)^j q^{i+\binom{j}{2}}}{(q;q)_{i-j}(q;q)_{j}(q^{2j+1}a^2;q)_{i-j}(q^{j}a^2;q)_{j}},
           \quad\mbox{if}\quad i\ge j,
\label{eq:g04-1}
\end{equation}
which can be directly confirmed. From (\ref{eq:F01}), (\ref{eq:g03}) and (\ref{eq:g04-1}) 
we therefore obtain the explicit form of $g_i$ as
\begin{equation*}
g_i  = \sum_{j=0}^i\frac{(-1)^j a q^{i+\binom{j+1}{2}}}{(q;q)_{i-j}(q;q)_{j}(q^{2j}a^2;q)_{i-j+1}(q^{j}a^2;q)_{j}}
        \sum_{m=0}^{2N}(-1)^m (\sigma_m-\sigma_{2N-m}) (q^ja)^{m-N} .
\end{equation*}
Therefore we obtain (\ref{gi_new}). 
\end{proof}

\begin{remark} 
Since $F(z)=F_+(z)+F_-(z)$ we can also express $g_i$ without using the elementary symmetric functions
$\sigma_{i}$, $i=1,\cdots,k$ as follows: 
\begin{equation*}
  g_i=\sum_{j=0}^i c_{ij}\,\big(F_+(q^{j}a)+F_-(q^{j}a)\big) ,
\end{equation*} 
which is somewhat simpler as one can see from the example of $ g_1 $ for the case $ N=3 $
\begin{equation*}
 g_1(a_i)= \displaystyle 
           \frac{q\prod_{\substack{1\le j\le 6 \\ j\ne i}}(1-a_ia_j)}{a_i^2(1-q)(1-qa_i^2)}
          -\frac{\prod_{\substack{1\le j\le 6 \\ j\ne i}}(1-qa_ia_j)}{qa_i^2(1-q)(1-q^2a_i^2)}
          +\frac{\prod_{\substack{1\le j\le 6 \\ j\ne i}}(a_j-qa_i)}{qa_i(1-qa_i^2)(1-q^2a_i^2)} ,
\end{equation*}
when compared to the middle coefficient in \eqref{N=3int_Recur:a}.
\end{remark}

\begin{remark}
The expression (\ref{gi_new}) for $g_i$ is indeed explicit, however as it arises from the inversion a matrix
it may yield a complicated expression whereas $g_{N-1}$ has the simple form
\begin{equation*}
  g_{N-1}=(-a)^{-N}q^{-\binom{N-1}{2}}(1-a_1a_2\cdots a_{2N}).
\end{equation*} 
\end{remark}

In addition to the linear $q$-difference equations given by either \eqref{I_q-DifferenceEqn} or \eqref{eq:g02-1}
we will derive a mixed, linear $q$-difference equation for $I_N$. 
\begin{theorem}\label{thm:C}
The integral $I_N$ satisfies the $(N-1)$-th order linear, partial $q$-difference equation
\begin{equation}
  I_N(a_1,\ldots,a_{2N}) = \sum_{i=1}^{N-1}b_i\, I_N(\ldots,q^{-1}a_i,\ldots) ,
\label{2nd_IN}
\end{equation}
where
\begin{equation}
  b_i = \frac{\prod_{m=N}^{2N}(1-q^{-1}a_ia_m)}{(1-q^{1-N}a_1a_2\cdots a_{2N})\prod_{\substack{1\le j\le N-1 \\ j\ne i}}(1-a_ia_j^{-1})} .
\label{eq:bi_defn}
\end{equation} 
Or equivalently, using the replacements $a_i \mapsto qa_i$ ($1\le i\le N-1$) in \eqref{2nd_IN}, then
\begin{equation}
 I_N( \underbrace{qa_1,\ldots,qa_{N-1}}_{\mbox{\tiny$q$\rm-shift}} ,a_N,\ldots,a_{2N})
  = \sum_{i=1}^{N-1}c_i\, I_N( \,\underbrace{qa_1,\ldots,qa_{i-1}}_{\mbox{\tiny$q$\rm-shift}},a_i,\underbrace{qa_{i+1},\ldots,qa_{N-1}}_{\mbox{\tiny$q$\rm-shift}} ,a_N,\ldots,a_{2N}),
\label{eq:IN-c} 
\end{equation}
where the coefficient $c_i$ is
\begin{equation}
  c_i = \frac{\prod_{m=N}^{2N}(1-a_ia_m)}{(1-a_1a_2\cdots a_{2N})\prod_{\substack{1\le j\le N-1 \\ j\ne i}}(1-a_ia_j^{-1})}.
\label{eq:coeff-c}
\end{equation} 
\end{theorem} 

\begin{remark} 
In particular, if $N=2$, then the above equation is
\begin{equation*}
  I_2(qa_1,a_2,a_3,a_4) = \frac{(1-a_1a_2)(1-a_1a_3)(1-a_1a_4)}{(1-a_1a_2a_3a_4)}I_2(a_1,a_2,a_3,a_4),
\end{equation*}
which is identical to \eqref{AWintegral}. Even for $N=3$ this equation is still relatively simple:
\begin{equation*}
 I_3(qa_1,qa_2,\ldots) = \frac{\prod_{m=3}^6(1-a_1a_m)}{(1-a_1a_2\cdots a_6)(1-a_1a_2^{-1})}I_3(a_1,qa_2,\ldots)
                        +\frac{\prod_{m=3}^6(1-a_2a_m)}{(1-a_1a_2\cdots a_6)(1-a_2a_1^{-1})}I_3(qa_1,a_2,\ldots). 
\end{equation*}
This can also be derived from \eqref{N=3int_Recur:a} along with another copy under $ a_1 \leftrightarrow a_2 $
and multiple use of the identity \eqref{Iidentity}. Equation \eqref{eq:IN-c} along with \eqref{eq:coeff-c} is
the simplest $q$-difference equation we can find as the coefficients are simple products.
\end{remark}

In order to prove Theorem~\ref{thm:C} using (\ref{eq:FR2}) of Lemma \ref{lem:FR} and (\ref{eq:T2}) we require the following result.
\begin{lemma}\label{lem:C}
The function $F(z)$ can be expanded as
\begin{equation}
  F(z) = C_0\prod_{i=1}^{N-1}\tilde{\phi}_1(z;a_i)+\sum_{j=1}^{N-1}C_j\prod_{\substack{1\le i\le N-1 \\ i\ne j}}\,\tilde{\phi}_1(z;a_i) ,
\label{eq:c01}
\end{equation}
where the coefficients are expressed as
\begin{equation}
  C_0 = \frac{(-1)^{N-1}}{a_1a_2\cdots a_{N-1}}(1-a_1a_2\cdots a_{2N})
\quad \mbox{\rm and} \quad 
  C_i = \frac{(-1)^{N}}{a_1a_2\cdots a_{N-1}}\frac{\prod_{m=N}^{2N}(1-a_ia_m)}{\prod_{\substack{1\le j\le N-1 \\ j\ne i}}(1-a_ia_j^{-1})} ,
\quad i = 1,\ldots, N-1 .  
\label{eq:c02}
\end{equation} 
\end{lemma}
\begin{proof} 
Given we have the expression (\ref{eq:c01}) and taking account of the degree of $F(z)$ as a Laurent polynomial,  
we can evaluate $C_i$. Comparing the coefficients of the highest-degree terms of both sides of (\ref{eq:c01}), we have
\begin{equation*}
  1-a_1a_2\cdots a_{2N}=C_0(-1)^{N-1}a_1a_2\cdots a_{N-1} ,
\end{equation*} 
which is equivalent to the expression (\ref{eq:c02}) for $C_0$. Next if we put $z=a_i$ ($i=1,2,\ldots,N-1$) and
use the vanishing property (\ref{eq:vanish}) we have  
\begin{equation}
  F_-(a_i)+F_+(a_i) = C_i\prod_{\substack{1\le j\le N-1 \\ j\ne i}}\,\tilde{\phi}_1(a_i;a_j).
\label{eq:c03}
\end{equation}
Since we have
\begin{equation*}
  F_-(a_i) = 0, \quad F_+(a_i) = \frac{\prod_{m=0}^{2N}(1-a_ia_m)}{a_i^{N-1}(1-a_i^2)} ,
\quad
  \tilde{\phi}_1(a_i;a_j) = (1-a_ja_i)(1-a_ja_i^{-1}) ,
\end{equation*} 
from (\ref{eq:c03}) we conclude that
\begin{equation*}
 \frac{\prod_{m=0}^{2N}(1-a_ia_m)}{a_i^{N-1}(1-a_i^2)} = C_i\prod_{\substack{1\le j\le N-1 \\ j\ne i}}(-a_ja_i^{-1})(1-a_ia_j)(1-a_ia_j^{-1}) ,
\end{equation*} 
which is equivalent to the expression (\ref{eq:c02}) for $C_i$. 
\end{proof}

As an application of Theorem \ref{thm:C} we shall show that $I_N$ satisfies a first-order vector-valued $q$-difference system. 
Let $T_{a}$ be the $q$-shift operator with respect to $a\mapsto qa$, i.e. $T_{a}f(a)=f(qa)$ for a function $f(a)$. 
\begin{corollary}\label{cor:1st}
Suppose that $\vec {\mathscr I}=({\mathscr I}_1,{\mathscr I}_2,\ldots,{\mathscr I}_{N-1})$ is the row-vector function defined by 
\begin{align*}
 {\mathscr I}_1 & := I_N(a_1,\underbrace{qa_2,\ldots,qa_{N-1}}_{\mbox{\tiny$q$\rm-shift}},a_N,\ldots,a_{2N}) ,
\\
 {\mathscr I}_i & := I_N(a_1,\underbrace{qa_2,\ldots,qa_{i-1}}_{\mbox{\tiny$q$\rm-shift}},a_i,
                             \underbrace{qa_{i+1},\ldots,qa_{N-1}}_{\mbox{\tiny$q$\rm-shift}},a_N,\ldots,a_{2N}) ,\quad 2\le i\le N-1 .
\end{align*}
Then $\vec {\mathscr I}$ satisfies the $q$-difference system
\begin{equation}
  T_{a_1}\vec {\mathscr I} = \vec {\mathscr I}\, A,
\label{eq:1st-diff}
\end{equation}
where $A$ is a $(N-1)\times (N-1)$ matrix defined by its Gauss decomposition
\begin{equation*}
A = \begin{pmatrix}
       c_1&   a_1a_2^{-1}&a_1a_3^{-1}&\cdots&a_1a_{N-1}^{-1}  \\
       &   d_{2}&&&   \\
       &   &d_{3}& &    \\
       &   & &\ddots&    \\
       &   & & &d_{N-1}\!\!    \\
    \end{pmatrix}
    \begin{pmatrix}
       \ 1&   &&&   \\
       c_2&  1 &&&   \\
       c_3&   &\ \ 1& &    \\
        \vdots &   & &\ddots &    \\
       c_{N-1} &   & & &\  1    \\
    \end{pmatrix} .
\end{equation*}
Here entries not shown are zero, 
while $c_i$ ($1\le i\le N-1$) are given by \eqref{eq:coeff-c} and 
$d_i$ $(2\le i\le N-1)$ are given by $d_i=\tilde{\phi}(a_i;a_1)=(1-a_1a_i)(1-a_1a_i^{-1})$. 
\end{corollary}

\begin{remark}
We immediately have the determinant of the coefficient matrix as 
\begin{equation}
 \det A = c_1d_2\ldots d_{N-1} = \frac{\prod_{m=2}^{2N}(1-a_1a_m)}{1-a_1a_2\cdots a_{2N}},
\label{eq:det}
\end{equation}
which shows that the system is non-degenerate if $a_1a_2\cdots a_{2N}\ne 1$ and $a_1a_m\ne 1$ for $2\le m\le 2N$.
Repeated use of (\ref{eq:det}) suggests that the $q$-Wronskian of the $q$-difference system \eqref{eq:1st-diff} for 
$ {\mathscr I} $ can be written as the $q$-Gamma product
\begin{equation*}
  \frac{(a_1\cdots a_{2N};q)_\8}{\prod_{1\le i<j\le 2N}(a_ia_j;q)_\8} ,
\end{equation*} 
up to some factor, which can be determined. If $N=2$ this is the well-known evaluation of the Askey-Wilson integral $I_2$.
\end{remark}

\begin{proof}
Multiplying (\ref{eq:1st-diff}) on the right by the inverse of the lower triangular matrix that occurs in $A$, we have
\begin{align}
  T_{a_1}\vec {\mathscr I}\,
  \begin{pmatrix}
    1&   &&&   \\
    -c_2&  1&&&   \\
    -c_3&   &\ \ 1& &    \\
    \ \vdots&   & &\ddots  &    \\
    \ \ -c_{N-1} &   & & &\  1    \\
  \end{pmatrix}
=
  \vec {\mathscr I}\,
  \begin{pmatrix}
    c_1&   a_1a_2^{-1}&a_1a_3^{-1}&\cdots&a_1a_{N-1}^{-1}  \\
    &   d_{2}&&&   \\
    &   &d_{3}& &    \\
    &   & &\ddots&    \\
    &   & & &d_{N-1}     \\
  \end{pmatrix},
\label{eq:matrices}
\end{align}
i.e. the result we set out to prove. 
The first columns of the matrices in \eqref{eq:matrices} establish equation (\ref{eq:IN-c}) in Theorem~\ref{thm:C}.
The $i$-th columns $(2\le i\le N-1)$ of the matrices in \eqref{eq:matrices} verify the trivial equation (\ref{eq:T3}) for $I_N$. 
\end{proof}

\section{An evaluation of \texorpdfstring{$I_N$}{IN} using \texorpdfstring{$BC_1$}{BC1}-type Jackson integrals}\label{BC1evalIN}
\setcounter{equation}{0}

Our final task is to give an evaluation of $ I_{N} $ in terms of standard basic hypergeometric functions
which is a natural generalisation of the $ N=2 $ case and in some senses is the simplest and minimal 
representation of the integral.
From its definition $\Phi(z)$ can be written as 
\begin{equation*}
  \Phi(z)=P(z)Q(z) ,
\end{equation*}
where
\begin{equation*}
  P(z) := \frac{z^{N/2}\theta(z^{-N};q^{N/2})}{\prod_{j=1}^{2N}z^{1/2-s_j}\theta(a_jz^{-1};q)} ,
\quad \mbox{and} \quad 
  Q(z) := (z^{-1}-z)\prod_{j=1}^{2N}z^{1/2-s_j}\frac{(qa_j^{-1}z;q)_\8}{(a_jz;q)_\8},
\end{equation*} 
and where $q^{s_i}=a_i$. The function $P(z)$ is invariant under the shift $ z\mapsto qz $, which is confirmed 
from the quasi-periodicity of $\theta(z;q)$ as $\theta(qz;q)=-\theta(z;q)/z$. 

\begin{lemma}\label{lem:Ie}
Under the condition
\begin{equation}
   \left| q^{N-1} \right| < |a_1a_2\cdots a_{2N}| ,
\label{eq:condition01}
\end{equation}
the following holds:
\begin{equation}
  I_{\varepsilon}:=\int_{|z|=\varepsilon}\Phi(z)\frac{dz}{2\pi \sqrt{-1}z}\to 0 \quad  (\varepsilon \to 0).
\label{eq:Ie}
\end{equation}
\end{lemma}
\begin{proof} 
We set $\varepsilon=|q|^M\varepsilon'$ for fixed $\varepsilon'>0$ and positive integer $M$. 
We will prove that $I_\varepsilon\to 0$ if $M\to +\8$. 
Since $P(z)$ is a $q$-periodic function on the compact set $|z|=|q|^M\varepsilon'$, $|P(z)|$ is bounded. 
In addition, $|Q(z)/z^{N-1-s_1-\cdots-s_{2N}}|$ is also bounded because
\begin{equation*}
  \frac{Q(z)}{z^{N-1-s_1-\cdots-s_{2N}}} = (1-z^2)\prod_{j=1}^{2N}\frac{(qa_j^{-1}z;q)_\8}{(a_jz;q)_\8}\to 1\quad (z\to 0) .
\end{equation*} 
Thus there exists $C>0$ independent of $M$ such that 
\begin{equation*}
  |\Phi(z)| = |P(z)Q(z)| < C \left| z^{N-1-s_1-\cdots-s_{2N}} \right| .
\end{equation*} 
Since $|q^{N-1}a_1^{-1}a_2^{-1}\cdots a_{2N}^{-1}|<1$ from Condition (\ref{eq:condition01}), 
putting $z=\varepsilon e^{2\pi \sqrt{-1}\zeta}$ we obtain 
\begin{align*}
 |I_{\varepsilon}| & < \int_0^1\left| \Phi \left(q^{M}\varepsilon' e^{2\pi \sqrt{-1}\zeta} \right) \right|d\zeta
\\
  & < C\int_0^1 \left| \left(q^{M}\varepsilon' e^{2\pi \sqrt{-1}\zeta} \right)^{N-1-s_1-\cdots-s_{2N}} \right|d\zeta
\\
  & < \left| q^{N-1}a_1^{-1}a_2^{-1}\cdots a_{2N}^{-1} \right|^M 
      C \int_0^1 \left| \left(\varepsilon' e^{2\pi \sqrt{-1}\zeta} \right)^{N-1-s_1-\cdots-s_{2N}} \right| d\zeta
  \to 0 \quad (M\to +\8),
\end{align*}
which establishes (\ref{eq:Ie}). 
\end{proof}

Taking account of Lemma \ref{lem:Ie}, under the Condition (\ref{eq:condition01}), $I_N$ can be written as the 
sum of residues of poles inside of $\mathbb{T}$ thus
\begin{equation*}
  I_N = \sum_{k=1}^{2N}\sum_{\nu=0}^\8 \Res_{z=a_kq^\nu}\Big[P(z)Q(z)\Big]\frac{dz}{z} .
\end{equation*} 
Since $P(z)$ is invariant under the shift $z\mapsto qz$, we have 
\begin{equation}
  I_N = \sum_{k=1}^{2N}\Big(\Res_{z=a_k}P(z)\frac{dz}{z}\Big)\sum_{\nu=0}^\8Q(a_kq^\nu) .
\label{eq:IN01}
\end{equation}
Here % $\Res_{z=a_k} P(z)\frac{dz}{z}$ is written as 
\begin{align}
 \Res_{z=a_k} P(z)\frac{dz}{z} & = \Res_{z=a_k} \frac{z^{N/2}\theta(z^{-N};q^{N/2})}{\prod_{j=1}^{2N}z^{1/2-s_j}\theta(a_jz^{-1};q)} \frac{dz}{z} ,
\notag \\
  & = \Res_{z=a_k} \frac{z^{-N/2+s_1+\cdots+s_{2N}}\theta(z^{-N};q^{N/2})}{\prod_{j=1}^{2N}\theta(a_jz^{-1};q)} \frac{dz}{z} ,
\nonumber \\
  & = \frac{a_k^{-N/2+s_1+\cdots+s_{2N}}\theta(a_k^{-N};q^{N/2})}{(q;q)_\8^2\prod_{\substack{1\le j\le 2N \\ j\ne k}}\theta(a_j/a_k;q)} ,
\quad \mbox{(See \cite[p.139, Lemma 3.1, (3.3)]{Ito_2006})} .
\label{eq:ResP}
\end{align}

On the other hand the contribution $\sum_{\nu=0}^\8Q(a_kq^\nu)$ is explained as follows. 
For an arbitrary $z\in \mathbb{C}^*$ we define the bilateral sum $J_N(z)$ by 
\begin{equation*}
   J_N(z):=\sum_{\nu=-\8}^\8Q(z q^\nu),
\label{eq:JNz}
\end{equation*}
which converges under the Condition (\ref{eq:condition01}). 
The sum $ J_N(z) $ is called the {\it $BC_1$-type Jackson integral} in the studies \cite{Ito_2009,ItoS_2008,Ito_2006}. 
In particular if we put $z=a_i$, $1\le i\le 2N$, then $J_N(a_i)$ is written as the unilateral sum
\begin{equation}
  J_N(a_i) = \sum_{\nu=0}^\8Q(a_iq^\nu),
\label{eq:JNa}
\end{equation}
because $Q(a_iq^\nu)=0$ if $\nu$ is an arbitrary negative integer. The sum $J_N(a_i)$ is called the {\it truncation} 
of $J_N(z)$. Moreover we define the function ${\mathcal J}_N(z)$ by
\begin{equation}
  {\mathcal J}_N(z) := \frac{J_N(z)}{h(z)} ,
\quad\mbox{where}\quad
  h(z) := z^{N-1-s_1-\cdots-s_{2N}}\frac{\theta(z^2;q)}{\prod_{j=1}^{2N}\theta(a_jz;q)} ,
\label{eq:JNr}
\end{equation}
which is called the {\it regularisation} of $J_N(z)$. Since the trivial poles and zeros of $J_N(z)$ are canceled 
out by multiplying $J_N(z)$ by $1/h(z)$ the function ${\mathcal J}_N(z)$ is holomorphic on $\mathbb{C}^*$, 
and ${\mathcal J}_N(z)$ is symmetric with respect to the reflection $z\to z^{-1}$. For readers familiar with 
the standard notation of the very-well-poised hypergeometric series we mention the correspondence between the 
$BC_1$-type Jackson integral ${\mathcal J}_N(z)$ and the very-well-poised $_{2N+2}\psi_{2N+2}$ hypergeometric series 
as follows:
\begin{equation}
  {\mathcal J}_N(z) = \frac{\prod_{i=1}^{2N}(qa_{i}^{-1}z)_{\infty }(qa_{i}^{-1}z^{-1})_{\infty}}{(qz^2)_{\infty }(qz^{-2})_{\infty}}
  \times{}_{2N+2}\psi_{2N+2}
\left[\!\!\begin{array}{c}
		qz,-qz,za_1,za_2,
		\ldots, za_{2N}\\
		z,-z,qz/a_{1},qz/a_{2},\ldots,qz/a_{2N} 
	  \end{array}\!\! ;q,\frac{q^{N-1}}{a_{1}a_2\cdots a_{2N}}
\right] .
\label{eq:2N+2psi}
\end{equation} 
In particular
\begin{equation}
  {\mathcal J}_N(a_1) = \frac{\prod_{i=1}^{2N}(qa_{1}a_{i}^{-1})_{\infty }(qa_{1}^{-1}a_{i}^{-1})_{\infty}}{(qa_{1}^2)_{\infty }(qa_{1}^{-2})_{\infty}}
\times{}_{2N+2}\varphi_{2N+1}
\left[\!\!\begin{array}{c}
		qa_1,-qa_1,a_1^2,a_1a_2,
		\ldots, a_1a_{2N}\\
		a_1,-a_1,qa_1/a_{2},\ldots,qa_1/a_{2N} 
	  \end{array}\!\! ;q,\frac{q^{N-1}}{a_{1}a_2\cdots a_{2N}}\right] ,
\notag
\end{equation} 
where the standard definitions \eqref{psiDefn}, \eqref{phiDefn} have been employed.
Now we return to Eq. (\ref{eq:IN01}). The following proposition sums up the argument made using 
(\ref{eq:IN01}), (\ref{eq:ResP}), (\ref{eq:JNa}) and (\ref{eq:JNr}).
\begin{proposition} 
\label{prop:IN-A}
Under the Condition {\rm (\ref{eq:condition01})} $I_N$ can be written in terms of ${\mathcal J}_N(a_k)$, $1\le k\le 2N$, by
\begin{gather}
  I_N = \sum_{k=1}^{2N}R_{k}\,{\mathcal J}_N(a_k) ,
\label{eq:IN02}
\end{gather}
where % $R_k$ is written as 
\begin{equation}
 R_k = \Big(\Res_{z=a_k}P(z)\frac{dz}{z}\Big)h(a_k) =
       \frac{a_k^{N/2-1}\theta(a_k^{-N};q^{N/2})}{(q;q)_\8^2\prod_{\substack{1\le j\le 2N \\ j\ne k}}\theta(a_ja_k;q)\theta(a_j/a_k;q)}.
\label{eq:Rk}
\end{equation}
\end{proposition}
The Jackson integral ${\mathcal J}_N(z)$ satisfies the following relation which is called the {\it Sears--Slater transformation}.
\begin{lemma}\label{lem:SS}
In our context the Sears--Slater transformation states 
\begin{equation}
  {\mathcal J}_N(z) = \sum_{i=1}^{N-1}{\mathcal J}_N(a_i)\prod_{\substack{1\le j\le N-1 \\ j\ne i}}
                      \frac{\theta(a_jz;q)\theta(a_j/z;q)}{\theta(a_ja_i;q)\theta(a_j/a_i;q)} .
\label{eq:SS1}
\end{equation}
\end{lemma}
\begin{proof} 
See \cite[p.247, Theorem 1.1]{ItoS_2008}. 
\end{proof}

\begin{remark}
As is explained in \cite{ItoS_2008} the expression given by \eqref{eq:SS1} is much simpler than the original 
expression of Sears and Slater who used the theory of very-well-poised hypergeometric series. In addition 
the expression (\ref{eq:SS1}) can be regarded as a connection formula, i.e. a general solution of the 
$q$-difference equation satisfied by ${\mathcal J}_N(z)$ which can be expressed as a linear combination 
of $N-1$ independent special solutions whose coefficients are $q$-gamma functions. That is why we use 
${\mathcal J}_N(z)$ instead of $_{2N+2}\psi_{2N+2}$ despite the slight difference in their expressions as we 
see in (\ref{eq:2N+2psi}). We will apply formula (\ref{eq:SS1}) in the next proposition in order to see that 
$I_N$ can be written as a sum of $(N-1)$ ${\mathcal J}_N(a_k)$'s.
\end{remark}

\begin{proposition}\label{prop:IN-B}
Under the Condition {\rm (\ref{eq:condition01})} $I_N$ can be written in terms of ${\mathcal J}_N(a_k)$, 
$1\le k\le N-1$, through
\begin{equation}
  I_N = \sum_{i=1}^{N-1}\bigg( R_{i} + \sum_{k=N}^{2N}R_{k}A_{ki} \bigg)\,{\mathcal J}_N(a_i) , 
\label{eq:IN03}
\end{equation} 
where $R_i$ is given by {\rm (\ref{eq:Rk})} and $A_{ki}$ by
\begin{equation}
  A_{ki} = \prod_{\substack{1\le j\le N-1 \\ j\ne i}}\frac{\theta(a_ja_k;q)\theta(a_j/a_k;q)}{\theta(a_ja_i;q)\theta(a_j/a_i;q)} . 
\label{eq:Aki}
\end{equation} 
\end{proposition}

\begin{remark} 
If we write $R_kA_{ki}$ in (\ref{eq:IN03}) explicitly then  
\begin{equation*}
  R_kA_{ki} = \frac{a_k^{N/2-1}\theta(a_k^{-N};q^{N/2})}{(q;q)_\8^2\theta(a_ia_k;q)\theta(a_i/a_k;q)
              \prod_{\substack{1\le j\le N-1 \\ j\ne i}}\theta(a_ja_i;q)\theta(a_j/a_i;q)
              \prod_{\substack{N\le j\le 2N \\ j\ne k}}\theta(a_ja_k;q)\theta(a_j/a_k;q)} .
\label{eq:RkAki}
\end{equation*}
\end{remark}

\begin{proof}
Putting $z=a_k$ in (\ref{eq:SS1}) of Lemma \ref{lem:SS} we have 
\begin{equation}
  {\mathcal J}_N(a_k) = \sum_{i=1}^{N-1}A_{ki}{\mathcal J}_N(a_i) , 
\label{eq:SS2}
\end{equation}
where $A_{ki}$ is defined by (\ref{eq:Aki}). From (\ref{eq:IN02}) and (\ref{eq:SS2}) we obtain 
\begin{align*}
  I_N & = \sum_{i=1}^{N-1}R_{i}\,{\mathcal J}_N(a_i)+\sum_{k=N}^{2N}R_{k}\,{\mathcal J}_N(a_k) ,
\\
      & = \sum_{i=1}^{N-1}R_{i}\,{\mathcal J}_N(a_i)+\sum_{k=N}^{2N}R_{k}\sum_{i=1}^{N-1}A_{ki}{\mathcal J}_N(a_i) ,
\\
      & = \sum_{i=1}^{N-1}{\mathcal J}_N(a_i)\Big( R_{i} + \sum_{k=N}^{2N}R_{k}A_{ki} \Big),
\end{align*}
which completes the proof. 
\end{proof}

Proposition~\ref{prop:IN-B} is the simplest evaluation we can show at present. A natural question is whether the 
coefficient $ R_{i}+\sum_{k=N}^{2N}R_{k}A_{ki} $ is a product of theta functions or not. As we know this is 
indeed a product if $N=2$, i.e. the Askey--Wilson case. As a contribution to an answer to this question we 
would like to provide another proof of the product formula for the Askey-Wilson integral $I_2$ which is not 
known from any other studies, which however does provide some insight for the cases $ I_N $ $ (N\ge 3)$. 
In pursuit of this goal we will establish a relation for theta functions as follows:
\begin{lemma}\label{lem:theta0}
For the theta function $\theta(z;q)$ we have  %the following relation holds: 
\begin{equation}
  \sum_{k=1}^4\frac{\theta(a_k^{-2};q)}{\prod_{\substack{1\le i\le 4 \\ i\ne k}}\theta(a_ia_k;q)\theta(a_i/a_k;q)}
   = \frac{2\,\theta(a_1a_2a_3a_4;q)}{\prod_{1\le i<j\le 4}\theta(a_ia_j;q)}.
\label{eq:theta}
\end{equation}
\end{lemma}
\begin{proof} See Corollary~\ref{cor:theta} in the Appendix.
\end{proof}

If $N=2$ then $R_{i}+\sum_{k=N}^{2N}R_{k}A_{ki}$ is just $R_1+R_2+R_3+R_4$ allowing us to write (\ref{eq:IN03}) as
\begin{equation*}
  I_2 = (R_1+R_2+R_3+R_4){\mathcal J}_2(a_1) .
\end{equation*} 
The expression (\ref{eq:theta}) applies to this situation because $R_k$ can be written as 
\begin{equation*}
  R_k = \frac{\theta(a_k^{-2};q)}{(q;q)_\8^2\prod_{\substack{1\le i\le 4 \\ i\ne k}}\theta(a_ia_k;q)\theta(a_i/a_k;q)} .
\end{equation*} 
From (\ref{eq:theta}) we have
\begin{equation*}
  \sum_{k=1}^4 R_k = \frac{2\,\theta(a_1a_2a_3a_4;q)}{(q;q)_\8^2\prod_{1\le i<j\le 4}\theta(a_ia_j;q)} .
\end{equation*} 
Therefore we obtain a relation between $I_2$ and ${\mathcal J}_2(a_1)$\footnote{Actually ${\mathcal J}_2(a_1)$
is given by Jackson's $_6\varphi_5$ evaluation.}
\begin{equation*}
  I_2 = \frac{2\,\theta(a_1a_2a_3a_4;q)}{(q;q)_\8^2\prod_{1\le i<j\le 4}\theta(a_ia_j;q)}{\mathcal J}_2(a_1) .
\end{equation*} 
Obtaining the relation between $I_2$ and ${\mathcal J}_2(a_1)$ is adequate for our purposes.

\section*{Acknowledgments}
The first author has been supported by the Japan Society for the Promotion of Science‎ KAKENHI Grant Number 
25400118 whilst the second author has been supported by the Australian Research Council.
The second author wishes to acknowledge Mizan Rahman for advice and encouragement and to thank
Filipo Colomo for the translations of an article by Pincherle. Assistance in preparing this manuscript was
provided by Jason Whyte.

\begin{appendix}
\section*{Appendix: Discussion of Proposition \ref{prop:IN-B} and a proof for Lemma \ref{lem:theta0}}\label{Appendix}
\setcounter{section}{1}
\setcounter{equation}{0}
\renewcommand{\thecorollary}{\Alph{section}.\arabic{corollary}}
\renewcommand{\thelemma}{\Alph{section}.\arabic{lemma}}
\renewcommand{\theproposition}{\Alph{section}.\arabic{proposition}}
\renewcommand{\theremark}{\Alph{section}.\arabic{remark}}
\setcounter{lemma}{0}
\setcounter{corollary}{0}
\setcounter{proposition}{0}

In this section we will give a tentative answer to the discussions in \S~\ref{BC1evalIN}. 
Proposition \ref{prop:IN-B} can be rewritten in the following way.
\begin{proposition}\label{prop:IN-C}
Consider a symmetric holomorphic function $f$ of $N+2$ variables ($x_1,x_2,\ldots,x_{N+2}$) on $(\mathbb{C}^*)^{N+2}$ defined 
by the $q$-difference equations
\begin{equation}
  f(\ldots,qx_i,\ldots) = (-1)^{N+1}\frac{f(x_1,x_2,\ldots,x_{N+2})}{(x_1x_2\cdots x_{N+2})x_i^{N-2}} ,
  \quad i=1,2,\ldots,N+2 .
\label{eq:A01}
\end{equation}
Under Condition (\ref{eq:condition01}) we can write $I_N$ in terms of ${\mathcal J}_N(a_k)$, $1\le k\le N-1$,
\begin{equation}
  I_N = \sum_{i=1}^{N-1} C_{i} {\mathcal J}_N(a_i) ,
\label{eq:IN04}
\end{equation}
where the connection coefficient $C_i$ ($i=1,\ldots,N-1$) is given by
\begin{equation}
  C_{i} = \frac{f(a_i,a_N,a_{N+1},\cdots ,a_{2N})}{(q;q)_\8^2\prod_{\substack{1\le j\le N-1 \\ j\ne i}}\theta(a_ja_i;q)\theta(a_j/a_i;q)
          \times \prod_{k=N}^{2N}\theta(a_ia_k;q)\times \prod_{N\le l<m\le 2N}\theta(a_la_m;q)} .
\label{eq:C01}
\end{equation} 
\end{proposition}

\begin{remark}
If $N=2$ then $f$ is uniquely determined as 
\begin{equation}
  f(x_1,x_2,x_3,x_4) = 2\theta(x_1x_2x_3x_4;q) .
\label{eq:A02}
\end{equation}
In this case (\ref{eq:A02}) permits us to rewrite (\ref{eq:IN04}) as
\begin{equation*}
  I_2 = \frac{2\,\theta(a_1a_2a_3a_4;q)}{(q;q)_\8^2\prod_{1\le j<k\le 4}\theta(a_ja_k;q)}\,{\mathcal J}_2(a_1) ,
\end{equation*} 
which is known as the relation between the Askey--Wilson integral $I_2$ and Jackson's very-well-poised $_6\varphi_5$ 
sum ${\mathcal J}_2(a_1)$
\begin{equation*}
   I_2 = \frac{2\,(a_1a_2a_3a_4;q)_\8}{(q;q)_\8 \prod_{1\le j<k\le 4}(a_ja_k;q)_\8} ,
\quad
  {\mathcal J}_2(a_i) = \frac{(q;q)_\8\prod_{1\le j<k\le 4}(qa_j^{-1}a_k^{-1};q)_\8}{(qa_1^{-1}a_2^{-1}a_3^{-1}a_4^{-1};q)_\8} .
\end{equation*} 
\end{remark}

\begin{proof}
From Proposition \ref{prop:IN-B} the connection coefficient $C_i$ in (\ref{eq:IN04}) can be written as
\begin{equation}
  C_i = R_{i} + \sum_{k=N}^{2N}R_{k}A_{ki} , 
\label{eq:C02}
\end{equation}
where $R_k$ and $A_{ki}$ are defined by (\ref{eq:Rk}) and (\ref{eq:Aki}) respectively. Without loss of generality
we will prove (\ref{eq:C01}) for the case $i=1$, i.e. we will evaluate $C_1$. In this case, from (\ref{eq:Rk}) 
and (\ref{eq:Aki}), the $R_kA_{ki}$ term in (\ref{eq:C02}) is
\begin{equation*}
  R_kA_{k1} = \frac{a_k^{N/2-1}\theta(a_k^{-N};q^{N/2})}{(q;q)_\8^2\theta(a_1a_k;q)\theta(a_1/a_k;q)
              \prod_{j=2}^{N-1}\theta(a_ja_1;q)\theta(a_j/a_1;q)
              \prod_{\substack{N\le j\le 2N \\ j\ne k}}\theta(a_ja_k;q)\theta(a_j/a_k;q)} .
\end{equation*} 
Using this and (\ref{eq:Rk}) in (\ref{eq:C02}) we obtain the explicit expression of $C_1$ as 
\begin{multline}
  C_1 = \frac{1}{(q;q)_\8^2\prod_{j=2}^{N-1}\theta(a_ja_1;q)\theta(a_j/a_1;q)}
\\ \times
        \Bigg[  \frac{a_1^{N/2-1}\theta(a_1^{-N};q^{N/2})}{\prod_{j=N}^{2N}\theta(a_ja_1;q)\theta(a_j/a_1;q)}  
                + \sum_{k=N}^{2N}\frac{a_k^{N/2-1}\theta(a_k^{-N};q^{N/2})}{\theta(a_1a_k;q)\theta(a_1/a_k;q)
                  \prod_{\substack{N\le j\le 2N \\ j\ne k}}\theta(a_ja_k;q)\theta(a_j/a_k;q)}
        \Bigg].
\label{eq:C03}
\end{multline}
Defining the function of $N+2$ variables $R$ as %  $R(x_1,x_2,\ldots,x_{N+2})$ 
\begin{equation*}
  R(x_1,x_2,\ldots,x_{N+2}) := \sum_{k=1}^{N+2}\frac{x_k^{N/2-1}\theta(x_k^{-N};q^{N/2})}
                                                    {\prod_{\substack{1\le i\le N+2 \\ i\ne k}}\theta(x_ix_k;q)\theta(x_i/x_k;q)},
\label{eq:R01}
\end{equation*}
allows us to rewrite (\ref{eq:C03}) as
\begin{equation*}
  C_1 = \frac{R(a_1,a_N,a_{N+1},\ldots,a_{2N})}{(q;q)_\8^2\prod_{j=2}^{N-1}\theta(a_ja_1;q)\theta(a_j/a_1;q)} .
\end{equation*} 
Therefore (\ref{eq:C01}) for $i=1$ is proved by confirming that
\begin{equation}
  R(x_1,x_2,\ldots,x_{N+2}) = \frac{f(x_1,x_2,\ldots,x_{N+2})}{\prod_{1\le i<j\le N+2}\theta(x_ix_j;q)} ,
\label{eq:R02}
\end{equation}
where the function $f(x_1,x_2,\ldots,x_{N+2})$ satisfies (\ref{eq:A01}). 
\end{proof}

We restate (\ref{eq:R02}) as in Lemma \ref{lem:theta} below.
\begin{lemma}\label{lem:theta}
There exists a symmetric holomorphic function $f$ %$f(x_1,x_2,\ldots,x_{N+2})$ 
on $(\mathbb{C}^*)^{N+2}$ such that 
\begin{equation*}
  \sum_{k=1}^{N+2}\frac{x_k^{N/2-1}\theta(x_k^{-N};q^{N/2})}{\prod_{\substack{1\le i\le N+2 \\ i\ne k}}\theta(x_ix_k;q)\theta(x_i/x_k;q)}
      = \frac{f(x_1,x_2,\ldots, x_{N+2})}{\prod_{1\le i<j\le N+2}\theta(x_ix_j;q)} . 
\label{eq:theta02}
\end{equation*}
In particular from the quasi-periodicity of $\theta(x;q)$ the function $f(x_1,x_2,\ldots,x_{N+2})$ satisfies {\rm (\ref{eq:A01})}.
\end{lemma}

In order to prove this lemma we will start from a slightly more general setting.
\begin{lemma}\label{lem:theta2}
Let $n$ be an integer satisfying $ n>2 $ and $g$ be an arbitrary holomorphic function of $x\in \mathbb{C}^*$. Then there exists 
a holomorphic function $h(x_1,x_2,\ldots,x_n)$ on $(\mathbb{C}^*)^n$ such that 
\begin{equation}
  \sum_{k=1}^{n}\frac{g(x_k)}{\prod_{\substack{1\le i\le n \\ i\ne k}}\theta(x_ix_k;q)\theta(x_i/x_k;q)}
   = \frac{h(x_1,x_2,\ldots, x_{n})}{\prod_{1\le i<j\le n}\theta(x_ix_j;q)} . 
\label{eq:theta03}
\end{equation}
\end{lemma}
\begin{proof}  
In order to see the symmetry of permutations of variables we use $ e(x;y):=x^{-1}\theta(xy;q)\theta(x/y;q) $
which has the property that $ e(x;y)=-e(y;x) $. Then Eq. (\ref{eq:theta03}) may be rewritten as 
\begin{equation}
  \prod_{1\le i<j\le n}e(x_i;x_j)\sum_{k=1}^n\frac{x_kg(x_k)}{\prod_{\substack{1\le l\le n \\ l\ne k}}e(x_l;x_k)}
       = \frac{h(x_1,x_2,\ldots,x_n)}{(x_1x_2\cdots x_n)^{n-2}}\prod_{1\le i<j\le n}x_j\theta(x_i/x_j;q). 
\label{eq:theta04}
\end{equation}
This equivalence means showing that if \eqref{eq:theta04} holds then also (\ref{eq:theta03}) holds.
The left-hand side of (\ref{eq:theta04}) is already known to be holomorphic because the factor 
$\prod_{\substack{1\le l\le n \\ l\ne k}}e(x_l;x_k)$ in the denominators of terms from the sum are canceled out by the 
common factor $\prod_{1\le i<j\le n}e(x_i;x_j)$. Moreover, since the left-hand side is skew symmetric with 
respect to permutations of variables, if we put $x_i=x_j$, then the left-hand side vanishes. This means the 
left-hand side of (\ref{eq:theta04}) is divisible by $x_j\theta(x_i/x_j;q)$, so is written as
\begin{equation*}
  h'(x_1,\ldots,x_n)\prod_{1\le i<j\le n}x_j\theta(x_i/x_j;q) ,
\end{equation*} 
where $ h'(x_1,\ldots,x_n) $ is some symmetric holomorphic function on $(\mathbb{C}^*)^n$. Therefore, defining 
$h(x_1,\ldots,x_n) := (x_1\cdots x_n)^{n-2}h'(x_1,\ldots,x_n)$ gives the right-hand side of (\ref{eq:theta04}). 
\end{proof}

\begin{proof}[Proof of Lemma \ref{lem:theta}.] 
Setting $n=N+2$, and $g(x)=x^{N/2-1}\theta(x^{-N};q^{N/2})$ in Lemma \ref{lem:theta2}, 
allows us to obtain $h(x_1,x_2,\ldots, x_{n})$ as $f(x_1,x_2,\ldots,x_{N+2})$ of Lemma \ref{lem:theta}.
\end{proof}

In general if a holomorphic function $g$ on $\mathbb{C}^*$ satisfies $g(qx)=-g(x)/x$ then $g$ is uniquely 
determined by $ g(x) = C \theta(x;q) $ for some constant $C$. Thus if the symmetric holomorphic function 
$f(x_1,x_2,\ldots,x_n)$ on $(\mathbb{C}^*)^n$ satisfies 
\begin{equation}
   f(\ldots,qx_i,\ldots) = -\frac{f(x_1,x_2,\ldots,x_n)}{x_1x_2\cdots x_{n}} ,
\label{eq:A03}
\end{equation}
then $f(x_1,x_2,\ldots,x_n)$ is equal to $\theta(x_1x_2\cdots x_n)$ up to constant. Now, turning to Equation 
(\ref{eq:A01}) only the $N=2$ case satisfies Equation \eqref{eq:A03}. Therefore Lemma \ref{lem:theta} in the case
$N=2$ implies the following Corollary.
\begin{corollary}\label{cor:theta}
We have the identity
\begin{equation}
  \sum_{k=1}^4\frac{\theta(x_k^{-2};q)}{\prod_{\substack{1\le i\le 4 \\ i\ne k}}\theta(x_ix_k;q)\theta(x_i/x_k;q)}
  = 2\,\frac{\theta(x_1x_2x_3x_4;q)}{\prod_{1\le i<j\le 4}\theta(x_ix_j;q)} .
\label{eq:theta2}
\end{equation}
(The constant $2$ in (\ref{eq:theta2}) can be fixed by an appropriate evaluation).
\end{corollary}
At present we can only confirm that the coefficient $C_i$ is simple for the case where $N=2$, which leads us to the 
problem: determine $f(x_1,\ldots,x_{N+2})$ in Lemma \ref{lem:theta} explicitly by theta functions possessing higher 
degree, i.e. determine the symmetric holomorphic function $f(x_1,\ldots,x_{N+2})$ satisfying {\rm (\ref{eq:A01})} 
explicitly by theta functions.
\end{appendix}

\section*{References}
\bibliographystyle{elsarticle-harv}
\bibliography{moment,random_matrices,nonlinear,qTheory,integrable}

\end{document}